\documentclass[11pt]{article}

\usepackage[margin=23mm]{geometry}
\usepackage{amsmath,amssymb,amsthm}
\usepackage[colorlinks=true,linkcolor=blue,citecolor=blue,urlcolor=blue]{hyperref}
\usepackage[dvipsnames]{xcolor}
\usepackage{tikz}
\usepackage{graphicx}

\usetikzlibrary{decorations.pathreplacing,arrows.meta}

\newtheorem{theorem}{Theorem}[section]
\newtheorem{lemma}[theorem]{Lemma}

\newtheorem{proposition}[theorem]{Proposition}
\newtheorem{conjecture}[theorem]{Conjecture}

\theoremstyle{definition}
\newtheorem{definition}[theorem]{Definition}
\theoremstyle{remark}
\newtheorem{remark}[theorem]{Remark}

\newcommand{\calC}{\mathcal C}
\newcommand{\calF}{\mathcal F}
\newcommand{\calB}{\mathcal B}
\newcommand{\calA}{\mathcal A}
\newcommand{\calR}{\mathcal R}
\newcommand{\conv}{\operatorname{conv}}
\newcommand{\E}{\mathbb E}
\newcommand{\Prob}{\mathbb P}

\title{On colorful Helly numbers and the growth of Tverberg numbers}

\author{Chaya Keller\thanks{Institute for the Theory of Computing,
  Faculty of Computer and Information Science,
  Ben-Gurion University of the Negev,
  Be'er Sheva 84105 Israel, and School of Computer Science, Ariel University, Israel.
\texttt{chayak@bgu.ac.il}.}
	\qquad \qquad
	Skand Parvatikar\thanks{Department of Mathematics, Statistics, and Computer Science, University of Illinois Chicago, Chicago, IL 60607, USA. Email: \texttt{sparvati@asu.edu}.} \qquad \qquad Tuan Tran\thanks{School of Mathematical Sciences, University of Science and Technology of China, China. Supported by the Excellent
Young Talents Program (Overseas) of the National Natural Science Foundation of China under Grant No.
GG0010007003.}
}

\date{}

\begin{document}
\flushbottom
\maketitle

\begin{abstract}
We establish new bounds on colorful Helly and Tverberg numbers in abstract convexity spaces with Radon number $r$. We show that the colorful Helly number is at most $(r-1)2^r$, improving Holmsen–Lee's bound $r^{r^{O(\log r)}}$ and matching the known lower bound $2^{r-1}-1$ up to a factor $O(r)$. Furthermore, we prove that the Tverberg number is $r_k=O(r^{\lceil \log_2 r\rceil }k)$, reducing P\'{a}lv\"{o}lgyi's coefficient $r^{r^{r^{O(\log r)}}}$ to a quasipolynomial function of $r$. For $S_4$-separable spaces, we prove $r_k=O(r^2k)$. In the other direction, we construct a sequence of $S_4$-separable spaces with Tverberg number $r_k=\Omega(r^2k)$. This shows that the weak Eckhoff conjecture, asserting that $r_k=O(rk)$ in any abstract convexity space, fails already in separable spaces. The same construction refutes the analogue of Reay's conjecture in separable spaces. 
\end{abstract}

\section{Introduction}\label{sec:intro}

\subsection{Background}

\paragraph{Helly's theorem, Radon's theorem, and their generalizations.} 
Helly's, Radon's, and Tverberg's theorems are central results in discrete
geometry, governing intersections of convex sets and partitions of point
sets.
The classical \emph{Helly's theorem}, discovered in 1913, asserts that for any finite family of convex sets in $\mathbb{R}^n$, if any $n+1$ sets in the family have a common point, then all sets in the family have a common point. One of its best-known generalizations is the \emph{colorful Helly theorem} of Lov\'asz and B\'ar\'any~\cite{Barany82} which states that if $F_1,\ldots,F_{n+1}$ are finite families of convex sets in $\mathbb R^n$ and every \emph{rainbow} set consisting of one set from each family has a common point, then in some family, all sets have a common point. For a survey on 
Helly's theorem and its generalizations, see~\cite{BaranyKalai22}.

Radon's theorem~\cite{Radon}, published in 1921, links intersection
and partition problems and yields a standard proof of Helly's theorem. It says that any $n+2$ points in $\mathbb R^n$ can be split into two parts whose convex hulls intersect. 
\emph{Tverberg's theorem}~\cite{Tverberg} extends this from two parts to $k$ parts: any $(k-1)(n+1)+1$ points in $\mathbb R^n$ can be split into $k$ parts whose convex hulls intersect. For a survey on the extensive recent work on Tverberg's theorem and its generalizations, see~\cite{BaranySoberon18}. 
In 1979, Reay~\cite{Reay} conjectured that the same point threshold remains
necessary when common intersection is weakened to requiring that the
convex hulls of every $t$ parts intersect, even for $t=2$.
This conjecture
has been studied extensively and remains wide open;
see~\cite{AsadaCFHPSTW18,BaranyKalai22,BaranySoberon18,PerlesS16,Roudneff09}.

\paragraph{Abstract convexity spaces.} 
The connection between Helly's and Radon's theorems motivates a broader
program of determining which intersection and partition theorems follow
from the Radon property alone.
This program goes back to Levi~\cite{Levi}, who in 1951 introduced the
notion of \emph{abstract convexity spaces} and showed that Helly's theorem
is a purely combinatorial consequence of Radon's theorem.
An abstract convexity space is a pair $(X,\calC)$, where $X$ is a nonempty set and $\calC\subseteq2^X$ contains $\emptyset, X$ and is closed under arbitrary intersections. The members of $\calC$ are called \emph{the convex sets}, and the \emph{convex hull} $\conv(P)$ of a set $P\subseteq X$ is the intersection of all convex sets containing $P$. A \emph{halfspace} is a convex set $C \in \calC$ whose complement $X \setminus C$ is convex as well. A convexity space $(X,\calC)$ is called \emph{$S_4$ separable} if for any disjoint convex sets $A,B \in \calC$, there exist disjoint halfspaces $\bar{A},\bar{B}$ such that $A \subset \bar{A},B \subset \bar{B}$. It is called \emph{$S_3$ separable} if the same separation property holds for a point and a convex set disjoint from it (see~\cite[Chapter~1.3]{vandeVel}). Under the natural assumption that
singletons are convex, $S_4$ separability implies $S_3$ separability. 


All these theorems admit natural formulations in abstract convexity
spaces. The \emph{Helly number} $h(X,\calC)$ is the least integer $h$
such that every finite family of convex sets with empty intersection
contains a subfamily of at most $h$ members with empty intersection.
The \emph{Radon number} $r(X,\calC)$ is the least integer $r$ such that
any $r$ points can be partitioned into two parts whose convex hulls
intersect. In this notation, Levi's theorem states that
$h(X,\calC)<r(X,\calC)$ whenever the Radon number is finite.
The \emph{colorful Helly number} $h_c(X,\calC)$ is the least number of
color classes for which the colorful Helly assertion holds. Similarly,
the \emph{$k$th Tverberg number} $r_k(X,\calC)$ and the
\emph{$t$-wise $k$th Tverberg number} $r_{k,t}(X,\calC)$ are the least
numbers of points that force the corresponding partition properties.
When the space is clear from the context, we omit $(X,\calC)$ and write
$h,r,h_c,r_k$, and $r_{k,t}$.

A central driving force behind the research of abstract convexity spaces since the 70's has been the \emph{partition conjecture} of Calder~\cite{Calder} and Eckhoff~\cite{Eckhoff79,Eckhoff00}, also known
as \emph{Eckhoff's conjecture}. 
It asserts that $r_k\leq(r-1)(k-1)+1$ in every abstract convexity space. This would make the sharp Euclidean form of Tverberg's theorem
a purely combinatorial consequence of Radon's theorem.
Jamison~\cite{Jamison} showed in 1981 that 
$r_k\leq r^{\lceil\log_2 k\rceil}$, but the conjecture remained wide open. 

\paragraph{Recent works on the colorful Helly and Tverberg numbers.}
In recent years, numerous works studied Helly-type problems in abstract convexity spaces (see, e.g.,~\cite{HolmsenSurvey,HolmsenPatakova,MoranYehudayoff}). In particular, Holmsen and Lee~\cite{HolmsenLee} proved that $h_c \leq r^{r^{O(\log r)}}$ and that the \emph{fractional Helly number} $h_f$ (see definition in Section~\ref{sec:applications}) satisfies $h_f \leq h_c$. In the other direction, P\'alv\"olgyi~\cite{Palvolgyi} constructed an example in which $h_c=2^{r-1}-1$. 

In 2010, Bukh~\cite{Bukh} famously disproved Eckhoff's conjecture by constructing an abstract convexity space with $r=4$ and $r_k \geq 3(k-1)+2$. On the other hand, he proved 
$r_k=O_r(k^2 \log^2 k)$. B\'{a}r\'{a}ny and Sober\'{o}n~\cite{BaranySoberon18} asked whether the weaker version of the conjecture, $r_k=O(kr)$, holds. This question was restated in several papers, sometimes under the name \emph{weak Eckhoff conjecture} (see~\cite{BaranyKalai22,CHJL,HolmsenSurvey, KellerSmorodinsky,Palvolgyi}). A major step toward a positive answer was made in 2022 by P\'alv\"olgyi~\cite{Palvolgyi} who obtained the optimal linear growth in $k$, using the aforementioned bound of Holmsen and Lee~\cite{HolmsenLee} on the fractional Helly number. He showed that for any $r,k$, we have $r_k \leq r^{r^{r^{O(\log r)}}}k$. 

In the last year, several new upper bounds on $r_k$ were obtained under separability assumptions: Alon and Smorodinsky~\cite{AlonSmorodinsky} proved $r_k=O(rk^2 \log k)$ under an $S_4$ separability assumption, Keller and Smorodinsky~\cite{KellerSmorodinsky} improved the bound under the same assumption to $r_k=O(r^2 k \log k)$, and Cho, Holmsen, Jung and Liu~\cite{CHJL} proved the even stronger bound $r_k=O((r^2 \log r)k)$ under the essentially weaker $S_3$ separability assumption. While these works suggest that the weak Eckhoff's conjecture holds, at least for separable spaces, we show below that in fact, \emph{it does not}.

\subsection{Our results}

In this paper, we obtain upper and lower bounds on the colorful Helly number $h_c$, the Tverberg number $r_k$ and the $t$-wise Tverberg number $r_{k,t}$, in convexity spaces with Radon number $r$.

\subsubsection{Improved upper bound on the colorful Helly number}
Recall that the best known bound on $h_c$ in terms of $r$ is $h_c \leq r^{r^{O(\log r)}}$, due to Holmsen and Lee~\cite{HolmsenLee}. 
Our first result is an essentially sharp single exponential upper bound. 
 \begin{theorem}
\label{thm:intro-colorful}
Every convexity space with Radon number $r\geq3$ and Helly number $h$, has
colorful Helly number $h_c \leq h2^r  \leq (r-1)2^r$.
\end{theorem}
In view of P\'alv\"olgyi's construction \cite{Palvolgyi} of a general convexity space with $h_c=2^{r-1}-1$, this bound is sharp up to a factor of $O(r)$. 
By the aforementioned result of Holmsen and Lee~\cite{HolmsenLee}, Theorem~\ref{thm:intro-colorful} yields the same upper bound $h_f \leq (r-1)2^r$ on the \emph{fractional Helly number} in general convexity spaces. We note that recently, Holmsen and P\'{a}t\'{a}kova~\cite{HolmsenPatakova} obtained the stronger bound $h_f \leq 2^r$ under the $S_3$ separability assumption.

By general methods of Alon, B\'ar\'any, F\"uredi and Kleitman~\cite{ABFK} and of
Alon, Kalai, Matou\v{s}ek and Meshulam~\cite{AKMM}, our bound on $h_f$ yields improved quantitative bounds for the \emph{selection lemma}~\cite{Barany82}, the \emph{weak $\varepsilon$-net theorem}~\cite{BaranyFurediLovasz} and the \emph{$(p,q)$-theorem}~\cite{AlonKleitman}, in general abstract convexity spaces. The details of these applications (including the required definitions) are presented in Section~\ref{sec:applications}.

\subsubsection{Improved upper bound on the Tverberg number}
Recall that the best known bound on $r_k$ in terms of $r$ that has the correct linear order in $k$ is $r_k \leq r^{r^{r^{O(\log r)}}}k$, due to P\'alv\"olgyi~\cite{Palvolgyi}. 
Our next result is a quasipolynomial 
upper bound. 
\begin{theorem}
\label{thm:intro-radon}
Every convexity space with Radon number $r\geq3$ satisfies $r_k\leq O(r^{\lceil \log_2 r \rceil})k$ for every $k\geq2$.
\end{theorem}
This result follows from a stronger (but more complex) upper bound in terms of $h,h_c$ and the Tverberg number $r_h$ (see Theorem~\ref{thm:main}). The only term that precludes the bound from being polynomial in $r$ is $r_h$, for which we could not find a replacement to Jamison's classical bound $r_h \leq r^{\lceil \log_2 h \rceil}$. Thus, any improved bound on the specific Tverberg number $r_h$ will lead to the same improvement in our bound on $r_k$, for any $k \geq h$. 

Our proof uses the improved bound on $h_c$ (i.e., Theorem~\ref{thm:intro-colorful}), along with 
a variant of the Sauer-Shelah lemma~\cite{Sauer,Shelah} for partitions into many parts, due to Karpovsky and Milman~\cite{KarpovskyMilman}.

\subsubsection{Tight upper bound on the Tverberg number in separable spaces}
Recall that the best known bound on $r_k$ in $S_4$ separable convexity spaces is $O((r^2 \log r)k)$, obtained by Cho, Holmsen, Jung and Liu~\cite{CHJL} (in fact, their bound holds under the essentially weaker $S_3$ separability assumption).
We prove the following bound, which is tight, in view of Theorem~\ref{thm:intro-lower}(1) below. 
\begin{theorem}
\label{thm:intro-radon-S4}
Every $S_4$ separable convexity space with Radon number $r\geq3$ satisfies $r_k\leq O(r^2k)$ for every $k\geq2$.
\end{theorem}
Alongside the improved bound on $h_c$ (i.e., Theorem~\ref{thm:intro-colorful}),  
our proof uses a binary separation tree based on  separating halfspaces which allows us to save a factor of $O(\log r)$ in the Sauer-Shelah argument, thus obtaining the sharp upper bound. 
It seems that such a separation tree technique was not used before in the study of Helly-type problems; possibly it will find further uses.

\subsubsection{Improved upper bounds on the $t$-wise Tverberg number}
Our next results are upper bounds on $r_{k,t}$ in terms of $r$. 
\begin{theorem}\label{thm:intro-reay-general}
Let $(X,\calC)$ be a convexity space with Radon number $r$. The following hold.
\begin{enumerate}
    \item For any $2 \leq t \leq k$, 
\[
r_{k,t}(X,\calC) = O  \left(\left(r^{\lceil\log_2t\rceil}+\log\binom kt\right) \cdot k \right) .
\]    
In particular, $r_{k,2}(X,\calC) = O(k(r+\log k))$.

\item If, in addition, $(X,\calC)$ is $S_4$ separable, then 
\[
r_{k,t}(X,\calC)
 =O\!\left(\min \left\{ r^2,rt\log t, rt+\log\binom{k}{t} \right\} \cdot k \right).
\]
In particular, $r_{k,t}(X,\calC)=O_t(rk)$ for all $k$ and $r_{k,t}(X,\calC)=O(trk)$ for any $k \leq \exp(O(r))$. 
\end{enumerate}
\end{theorem}
The proof of~(1) uses the Sauer-Shelah type argument from the proof of Theorem~\ref{thm:intro-radon}. For $t=2$, this general-space bound is stronger than the bound $r_{k,2}=O(rk \log (rk))$ obtained in~\cite{KellerSmorodinsky} under an $S_4$ separability assumption.

The assertion of~(2) contains three different bounds. The first is $O(r^2k)$, which is simply the bound of Theorem~\ref{thm:intro-radon-S4}. The second is $O(rtk\log t)$, whose proof combines an $\varepsilon$-net argument of the type used by Cho et al.~\cite{CHJL} with an observation of Alon and Smorodinsky~\cite{AlonSmorodinsky} that uses $S_4$ separability to leverage a family of non-intersecting convex sets into a family of non-intersecting halfspaces containing them. 
The third is $O((rt+\log\binom{k}{t})k)$, whose proof uses the binary separation tree argument from the proof of Theorem~\ref{thm:intro-radon-S4}. For $k\leq \exp(O(r))$, the upper bound becomes
$O(rk\min\{t,r\})$. The examples of Theorem~\ref{thm:intro-lower} attain this
order of magnitude along a sequence of examples with $k=\exp(\Theta(r))$. For larger values of $k$, we weren't able to obtain a tight bound.

\subsubsection{Counterexample to the weak Eckhoff conjecture}
Our most surprising result is a construction which shows that the weak Eckhoff's conjecture fails even in $S_4$ separable convexity spaces and that the analogue of Reay's conjecture to separable convexity spaces fails badly as well. We prove the following:

\begin{theorem}
\label{thm:intro-lower}
For every sufficiently large $n$, there is a finite $S_4$ separable convexity space $(X_n,\calC)$ with Radon number $r=\Theta(\log n)$, such that:
\begin{enumerate}
    \item For any $64\log^2 r \leq k \leq n$, we have $r_k(X_n,\calC)=\Omega(rk\log k)$. Furthermore, $r_n(X_n,\calC)=\Theta(r^2 n)$.

    \item $r_{n,t}(X_n,\calC) = \Theta(rn \min(r,t))$. In particular, for any constant $t$, we have $r_{n,t}(X_n,\calC)=\Theta(rn)$.
\end{enumerate}
\end{theorem}
The first claim implies that the weak Eckhoff conjecture fails even in separable convexity spaces, and provides the lower bound $r_k=\Omega(r^2 k)$ for an infinite sequence of pairs $(r,k)$. In particular, it shows that the $r_k=O(r^2 k)$ upper bound in $S_4$ separable spaces obtained in Theorem~\ref{thm:intro-radon-S4} is sharp up to a constant factor, and that the 
$r_k=O((r^2 \log r)k)$ upper bound in $S_3$ separable spaces obtained by Cho et al.~\cite{CHJL} is sharp up to a factor of $O(\log r)$. 

The second claim implies that the analogue of Reay's conjecture to $S_4$ separable convexity spaces fails and provides a gap of factor $\Omega(r)$ between $r_k$ and $r_{k,2}$ for an infinite sequence of pairs $(r,k)$. In particular, it shows that the  $O(rk\min\{t,r\})$ upper bound in $S_4$ separable spaces of Theorem~\ref{thm:intro-reay-general} is attained
along the sequence of examples with $k=n=\exp(\Theta(r))$. 

The spaces $X_n$ are variants of \emph{box convexity} on the grid $[q]^n$. A counting argument bounds from above their Radon number, while a probabilistic argument provides the lower bounds on the Tverberg and the $t$-wise Tverberg numbers. 

\subsection{Organization of the paper} 
The paper is organized as follows. 
In Section~\ref{sec:colorful} we prove the new upper bound on the colorful Helly number (i.e.,  Theorem~\ref{thm:intro-colorful}). In Section~\ref{sec:upper} we  prove the new upper bound on the Tverberg number in general spaces (i.e., Theorem~\ref{thm:intro-radon}). 
In Section~\ref{sec:upper-separable} 
we  prove the new upper bound on the Tverberg number in $S_4$ separable spaces (i.e., Theorem~\ref{thm:intro-radon-S4}).
In Section~\ref{sec:lower} we present the construction that refutes the weak Eckhoff conjecture and prove  Theorem~\ref{thm:intro-lower}. In Section~\ref{sec:reay} we prove the new upper bounds on the $t$-wise Tverberg number (i.e., Theorem~\ref{thm:intro-reay-general}). In  Section~\ref{sec:applications} we briefly present applications of our results to improved quantitative bounds for the fractional Helly theorem, the selection lemma, the weak $\varepsilon$-net theorem, and the $(p,q)$ theorem. 
We conclude the paper with a few open problems in  Section~\ref{sec:conclusion}.

Throughout the paper, we state the results in the setting of multisets, where labeled copies of the same object are allowed, as this is 
useful when
counting convex hulls in  Section~\ref{sec:upper} and is needed for the weighted arguments in
Section~\ref{sec:applications}. Such `repetitions' affect counting arguments, but not convex hulls or intersections. Obviously, the upper bounds obtained in the multiset setting hold in the standard setting of sets. In the lower bound constructions, we prove that the elements of the multiset can be made distinct, and hence the constructions apply in the standard setting of sets as well.   We write $\log$ for the natural logarithm unless $\log_2$ is indicated explicitly.

\section{New Upper Bound on the Colorful Helly Number}\label{sec:colorful}

In this section we prove Theorem~\ref{thm:intro-colorful}, which asserts that in any abstract convexity space 
$(X,\calC)$ with Radon number $r \geq 3$ and Helly number $h$, the colorful Helly number is at most $h2^r$, which is at most $(r-1)2^r$, as by Levi's theorem~\cite{Levi}, $h\leq r-1$. 

\paragraph{Definitions and outline of the proof.}
Let $\calF_1,\ldots,\calF_m$ be finite nonempty labeled families of convex sets, regarded as $m$ \emph{color classes}. The family (of families) $(\calF_1,\ldots,\calF_m)$ is said to be \emph{$t$-rainbow} if every color class $\calF_i$
has empty total intersection, while every labeled subfamily containing at
most $t$ members of each color class has a common point. Thus, a
$1$-rainbow family is precisely a counterexample to the colorful Helly property.

The proof of the theorem is an iterative process that begins with a $1$-rainbow family and at the $t$'th step, `upgrades' a $t$-rainbow family into a $(t+1)$-rainbow family by removing several color classes. In this transition process, a labeled subfamily $T$ is called a \emph{witness} if $|T|\leq h$, it
contains at most $t+1$ members of any one color class, and the intersection of its elements is
empty. Every failure of the $(t+1)$-rainbow property contains such a witness, by the definition of the Helly number $h$. To make the transition, several color classes are removed so that no witnesses remain, and thus, the resulting family is $(t+1)$-rainbow. 

As by the definition of the Helly number, an $h$-rainbow family must be empty, the process terminates after at most $h-1$ steps, with all color classes removed. Hence, if we show that the process can be performed in such a way that only a few colors are removed at each step, this will provide an upper bound on the size of the initial $1$-rainbow family, which is an upper bound on the maximum size of a family that does not satisfy the colorful Helly property. The transition from a $t$-rainbow family to a $(t+1)$-rainbow family is demonstrated in Figure~\ref{fig:rainbow-step}.   

The core idea of our proof is similar to that of Holmsen and Lee~\cite{HolmsenLee}; we discuss the relation between the proofs at the end of the section.

\begin{figure}[tbp]               
 \centering 
 \resizebox{0.72\linewidth}{!}{%
   \begin{tikzpicture}[
  x=1cm,y=1cm,font=\small,line cap=round,line join=round,
  count/.style={decorate,decoration={brace,amplitude=3pt},line width=.5pt}
]
\definecolor{rainbowTeal}{RGB}{39,119,120}
\definecolor{rainbowRed}{RGB}{185,83,67}
\definecolor{rainbowViolet}{RGB}{112,89,151}
\definecolor{rainbowEdge}{RGB}{77,112,140}
\tikzset{intersection/.style={draw=rainbowEdge!85,line width=.7pt,
  fill=rainbowEdge!7}}

\node[font=\normalsize] at (2.00,4.48) {$t$-rainbow};
\node[font=\normalsize] at (7.475,4.48) {$(t+1)$-rainbow};

\path[draw=rainbowRed,line width=.75pt,
      dash pattern=on 2.6pt off 1.7pt,fill=rainbowRed!4]
  (-.02,3.04)
  .. controls (-.03,3.49) and (.30,3.70) .. (.73,3.64)
  .. controls (1.50,3.55) and (2.53,3.72) .. (3.29,3.64)
  .. controls (3.94,3.67) and (4.14,3.30) .. (4.02,2.73)
  .. controls (4.00,2.29) and (3.57,2.10) .. (3.18,2.18)
  .. controls (2.99,2.20) and (2.76,2.19) .. (2.60,2.16)
  .. controls (2.36,2.11) and (2.40,1.75) .. (2.16,1.71)
  .. controls (1.83,1.57) and (1.59,1.75) .. (1.56,1.99)
  .. controls (1.54,2.12) and (1.42,2.19) .. (1.22,2.17)
  .. controls (1.02,2.16) and (.84,2.12) .. (.63,2.17)
  .. controls (.03,2.09) and (-.11,2.59) .. (-.02,3.04) -- cycle;

\path[intersection]
  (.12,3.04)
  .. controls (.09,3.41) and (.40,3.55) .. (.74,3.50)
  .. controls (1.47,3.41) and (2.56,3.58) .. (3.27,3.50)
  .. controls (3.82,3.52) and (3.97,3.28) .. (3.88,2.76)
  .. controls (3.86,2.37) and (3.49,2.20) .. (3.18,2.28)
  .. controls (2.38,2.41) and (1.39,2.23) .. (.66,2.30)
  .. controls (.19,2.26) and (.02,2.62) .. (.12,3.04) -- cycle;

\path[intersection]
  (6.34,2.86)
  .. controls (6.33,3.35) and (6.56,3.58) .. (6.99,3.50)
  .. controls (7.41,3.42) and (7.82,3.54) .. (8.19,3.50)
  .. controls (8.65,3.50) and (8.70,3.13) .. (8.64,2.57)
  .. controls (8.77,2.06) and (8.55,1.65) .. (8.13,1.72)
  .. controls (7.69,1.80) and (7.35,1.68) .. (6.87,1.74)
  .. controls (6.38,1.70) and (6.27,2.21) .. (6.34,2.86) -- cycle;

\node[text=rainbowEdge] at (2.00,3.98) {$\bigcap S\ne\varnothing$};
\node[text=rainbowEdge] at (7.475,3.90) {$\bigcap S'\ne\varnothing$};
\node[text=rainbowRed] at (2.00,1.36)
  {$\bigcap\bigl(S\cup\{F\}\bigr)=\varnothing$};

\foreach \after/\i/\xx/\col in {
  0/1/.65/rainbowTeal,
  0/2/2.00/rainbowRed,
  0/3/3.35/rainbowViolet,
  1/1/6.80/rainbowTeal,
  1/3/8.15/rainbowViolet}
{
  \fill[\col] (\xx,3.18) circle[radius=2pt];
  \node[text=\col,font=\scriptsize,inner sep=0pt]
    at (\xx,2.875) {$\vdots$};
  \fill[\col] (\xx,2.57) circle[radius=2pt];
  \fill[\col] (\xx,2.02) circle[radius=2pt];
  \node[text=\col!65,font=\scriptsize,inner sep=0pt]
    at (\xx,1.03) {$\vdots$};
  \fill[\col!65] (\xx,.56) circle[radius=2pt];
  \node[text=\col] at (\xx,.15) {$\mathcal F_{\i}$};
}

\node[anchor=west,inner sep=1pt,text=rainbowRed]
  at (2.12,2.025) {$F$};

\draw[count,draw=rainbowEdge] (-.30,2.41) -- (-.30,3.34);
\node[anchor=east,text=rainbowEdge] at (-.46,2.875) {$t$};
\draw[count,draw=rainbowEdge] (8.93,3.34) -- (8.93,1.83);
\node[anchor=west,text=rainbowEdge] at (9.09,2.585) {$t+1$};

\draw[-{Stealth[length=2.5mm,width=1.8mm]},line width=.8pt]
  (4.30,2.58) -- (5.92,2.58);
\node[text=rainbowRed] at (5.11,2.97) {$-\,\mathcal F_2$};
\node[text=rainbowRed] at (2.00,-.22) {$\times$};

\end{tikzpicture}%
 }               
 \caption{Illustration of the transition from a $t$-rainbow family to a $(t+1)$-rainbow family, by removing the offending color $\calF_2$.\\
 }
 \label{fig:rainbow-step}   
\end{figure}


\medskip In order to bound the number of sets removed at each step, we use the following simple lemma.

\begin{lemma}
\label{lem:partition-cover}
Let $1\leq t\leq r-2$. There is a set of $2^{r-t}-1$ partitions of $[r]$ into
$t+1$ nonempty blocks such that every nontrivial bipartition of $[r]$ is
refined by at least one of them.
\end{lemma}

\begin{proof}
Fix $S\subseteq[r]$ of size $r-t+1$, and use the other $t-1$ elements as
singleton blocks in every partition. For each unordered nontrivial
bipartition $S=U\sqcup(S\setminus U)$, take these two blocks together with
the $t-1$ singleton blocks. This gives $2^{|S|-1}-1=2^{r-t}-1$
partitions.

Let $A\sqcup B=[r]$ be a nontrivial bipartition. If both $A$ and $B$ meet
$S$, the partition associated with $U=A\cap S$ refines $A\mid B$. If $S$
lies wholly on one side, any of the displayed partitions does so, since
every element outside $S$ is a singleton block.
\end{proof}

\begin{lemma}\label{lem:rainbow-step}
Let $1\leq t<h$, and suppose that the family
$(\calF_1,\ldots,\calF_m)$ of $m$ color classes is $t$-rainbow. There exists a collection of at most $(2^{r-t}-2)(h-t)$ color classes, such that after deleting them, 
the remaining family of color classes is either empty or
$(t+1)$-rainbow.
\end{lemma}

\begin{proof}
Choose a family of witnesses $T_1,\ldots,T_s$ which is maximal subject to no color occurring
in two different witnesses. We claim that 
\begin{equation}\label{Eq:Witnesses}
s\leq2^{r-t}-2.
\end{equation}
Assume on the contrary that $s \geq 2^{r-t}-1$. 
Retain $2^{r-t}-1$ witnesses and biject them with the partitions from
Lemma~\ref{lem:partition-cover}. Split every $T_i$ into $t+1$ columns
$T_i^1,\ldots,T_i^{t+1}$ so that each column contains at most one set of
each color; such a split exists because no color contributes more than
$t+1$ sets. Put
\[
 C_i^j=\bigcap_{F\in T_i^j}F,
\]
where the intersection of an empty column is $X$. Each $C_i^j$ is convex,
and
$\bigcap_{j=1}^{t+1}C_i^j=\emptyset$.

Write the partition assigned to $T_i$ as
$P_i^1\sqcup\cdots\sqcup P_i^{t+1}=[r]$. For each $a\in[r]$, perform the following process. Discard
from every witness $T_i$ the unique column $j$ with $a\in P_i^j$ and keep the
other $t$ columns. The union of all kept columns contains at most $t$ sets
of any color, because a column contains at most one set of a fixed color
and different witnesses use disjoint colors. It has a
common point by the $t$-rainbow property of the original family; choose one and call it $x_a$.

Apply the Radon property to $x_1,\ldots,x_r$. There is a nontrivial
bipartition $A\sqcup B=[r]$ and a point
\[
 x\in\conv\{x_a:a\in A\}\cap\conv\{x_b:b\in B\}.
\]
By Lemma~\ref{lem:partition-cover}, one of the $(t+1)$-partitions of $[r]$ assigned to the witnesses, say the one assigned to $T_i$,
refines $A\mid B$. If $P_i^j\subseteq A$, then column $j$ was kept in the
construction of every $x_b$ with $b\in B$. Hence $x_b\in C_i^j$ for all
$b\in B$, and convexity gives $x\in C_i^j$. The same argument with $A$
and $B$ interchanged applies when $P_i^j\subseteq B$. Thus
$x\in\bigcap_jC_i^j$, a contradiction. This proves~\eqref{Eq:Witnesses}.

Delete every color class met by $T_1,\ldots,T_s$. Note that every witness meets at most $h-t$
color classes: the $t$-rainbow property of the family forces one color to contribute $t+1$
members, and each further color contributes at least one.
Hence, overall at most
$(2^{r-t}-2)(h-t)$ colors are deleted. If the surviving family were
nonempty and fails to be $(t+1)$-rainbow, there would exist a
witness using only undeleted color classes, contradicting maximality of $(T_1,\ldots,T_s)$.
\end{proof}

Now we are ready to present the proof of Theorem~\ref{thm:intro-colorful}.
                
\begin{proof}[Proof of Theorem~\ref{thm:intro-colorful}]
If $h=1$, every finite family of nonempty convex sets has a common point.
No member of a colorful Helly counterexample can be empty, so no color
class can have empty total intersection; hence $h_c=1$. Assume $h\geq2$.
Starting from a $1$-rainbow
family, apply Lemma~\ref{lem:rainbow-step} successively for
$t=1,\ldots,h-1$. This is legitimate since $h\leq r-1$. The family left
at the end is $h$-rainbow and must be empty, by the definition of the
Helly number. Consequently, a colorful Helly counterexample has at most
$\sum_{t=1}^{h-1}(2^{r-t}-2)(h-t)$
color classes. Therefore,
\begin{align*}
h_c
\leq1+\sum_{t=1}^{h-1}(2^{r-t}-2)(h-t)
=1+2^r\bigl(h-2+2^{1-h}\bigr)-h(h-1)
 <2^rh\leq2^r(r-1),
\end{align*}
where the equality uses the simplification
$\sum_{t=1}^{h-1}(h-t)2^{-t}=h-2+2^{1-h}$. For $h=2$, the exact
expression is $h_c\leq2^{r-1}-1$.
\end{proof}

\begin{remark}
    Our proof follows the same general philosophy as the proof of Holmsen and Lee~\cite{HolmsenLee}: suitable partitions are encoded by color classes or color-disjoint witnesses, and a Radon/Tverberg-type partition of the resulting intersection points is used to obtain a contradiction. The main difference is that Holmsen and Lee make one global step, encoding all $h$-partitions of an $r_h$-element set, whereas we iteratively upgrade a $t$-rainbow family into a $(t+1)$-rainbow family, and thus, at each step we need only $2^{r-t}-1$ partitions that refine all Radon bipartitions of $[r]$. Thus, in our proof, repeated applications of the Radon property replace a much more expensive $h$-Tverberg step in the proof of Holmsen and Lee.
\end{remark}

\section{New Upper Bound on the Tverberg Number}\label{sec:upper}

In this section we prove the following upper bound on the Tverberg number in abstract convexity spaces in terms of the Radon number, which implies Theorem \ref{thm:intro-radon}. 
\begin{theorem}\label{thm:main}
Let $(X,\calC)$ be an abstract convexity space with Helly number $h\geq 2$ and colorful Helly number $h_c$. 
Then for every $k\geq 2$,
\begin{equation}\label{eq:main-explicit}
 r_k(X,\calC) \leq 20 k\bigl(r_h(X,\calC)+h\log(4h^2 h_c)\bigr).
\end{equation}
\end{theorem}
\begin{proof}[Proof of Theorem \ref{thm:intro-radon}]
To see that Theorem~\ref{thm:main} implies the $r_k(X,\calC) \leq O(r^{\lceil \log_2 r \rceil})k$ upper bound asserted in Theorem~\ref{thm:intro-radon}, note that if the Radon number of $(X,\calC)$ is $r$ then by Levi's theorem, $h<r$; by Jamison's theorem, $r_h(X,\calC)\leq r^{\lceil\log_2 h\rceil}\leq r^{\lceil \log_2 r \rceil}$; and by Theorem~\ref{thm:intro-colorful}, $h_c<h2^r$. Since $h<r$, the latter implies $h\log(h^2 h_c)=O(r^2) \leq O(r^{\lceil \log_2 r \rceil})$, where the inequality holds as $r > h \geq 2$.  Substituting these bounds into~\eqref{eq:main-explicit} yields $r_k(X,\calC) \leq 20k(r^{\lceil \log_2 r \rceil}+O(r^2))=O(r^{\lceil \log_2 r \rceil})k$.
\end{proof}


\paragraph{Definitions and overview of the proof of Theorem~\ref{thm:main}.} Throughout the proof, we write $r_h:=r_h(X,\calC)$. Let $N_k=20 k\bigl(r_h+h\log(4h^2 h_c)\bigr).$ Given a set $P \subset X$ of $N_k$ points, we color the points of $P$ in $2k$ colors, independently at random. We would like to show that there exists a coloring for which the convex hulls of some $k$ color classes have a non-empty intersection. This will imply the assertion of the theorem, as one can take such a coloring and complete the $k$ color classes to a $k$-Tverberg partition (i.e., a Tverberg partition into $k$ parts) by adding the remaining points to the $k$ color classes arbitrarily. 

In order to exploit Helly-type arguments, we consider sets of $h$ color classes. We say that a set of $h$ color classes is \emph{bad} if the convex hulls of those color classes have an empty intersection.

The proof has two steps. First, we show, using a variant of the Sauer-Shelah lemma for $h$-partitions and Binomial tail estimates, that for every fixed set of $h$ color classes, the probability that it is bad is very low. We conclude that there exists a coloring for which only a few of the sets of $h$ color classes are bad.

The second step is deterministic. We prove, using the colorful Helly theorem, that if a family of $2k$ convex sets has only very few non-intersecting $h$-subfamilies,
then at least $k$ of the sets have a common point. Taking the coloring found in the first step and applying this result to the convex hulls of the color classes
completes the proof.

The core idea of our proof is similar to that of P\'{a}lv\"{o}lgyi~\cite{Palvolgyi}; we discuss the relation between the proofs at the end of the section.

\paragraph{Step~1: Showing that there exists a coloring with very few bad sets of $h$ color classes.}

At this step, we prove the following proposition.
\begin{proposition}\label{prop:random-bins}
Let $(X,\calC)$ have Helly number $h\geq 2$, and put $r_h:=r_h(X,\calC)$. Let $P$ be an $N$-element multiset in $X$, and color each element of $P$ independently and uniformly in one of $2k$ colors. Denote the color classes by $P_1,\ldots,P_{2k}$. If $\eta\in(0,1)$ and $\frac{N}{2k}\geq 8\bigl(r_h+\log(1/\eta)\bigr)$,
then for every fixed set $I$ of $h$ of color classes,
\begin{equation}\label{Eq:Aux-Prop3.2}
    \Prob\left[\bigcap_{i\in I}\conv(P_i)=\emptyset\right]\leq\eta.
\end{equation}
Consequently, there exists a coloring for which at most $\eta\binom{2k}{h}$ of the sets of $h$ color classes have empty intersection.
\end{proposition}

We prove Proposition~\ref{prop:random-bins} by a Sauer-Shelah type argument, which uses the following variant of the Sauer-Shelah lemma for $h$-partitions, due to Karpovsky and Milman \cite{KarpovskyMilman}. 
\begin{theorem}[Sauer-Shelah for $h$-partitions~\cite{KarpovskyMilman}]\label{thm:qary-sauer}
Let $m,u \in \mathbb{N}$. Let $Q$ be an $m$-element set, let $h\geq 2$, and let $\calA\subseteq [h]^Q$ be a family of $h$-colorings of $Q$. Suppose that no $u$-element subset $A\subseteq Q$ is shattered by $\calA$, meaning that for every $A\in\binom{Q}{u}$, the set of restrictions $\{\chi|_A:\chi\in\calA\}$ misses at least one element of $[h]^A$. Then
\[
 |\calA|\leq \sum_{i=0}^{u-1}\binom{m}{i}(h-1)^{m-i}.
\]
\end{theorem}

The heart of our argument is the following lemma.

\begin{lemma}\label{lem:coloring-pattern}
Let $h\geq 2$ and let $r_h:=r_h(X,\calC)$. Let $Q$ be an $m$-element multiset in $X$. Color the elements of $Q$ independently and uniformly with colors in $[h]$, and let $Q_i$ be the sub-multiset of elements of color $i$. Then
\[
 \Prob\left[\bigcap_{i=1}^h \conv(Q_i)=\emptyset\right]
 \leq \Prob[\operatorname{Bin}(m,1/h)<r_h].
\]
\end{lemma}

\begin{proof}
By the definition of $r_h$, for every $r_h$-element sub-multiset $A\subseteq Q$, there exists a labeled $h$-Tverberg partition $A=A_1\sqcup\cdots\sqcup A_h$ with $\bigcap_{i=1}^h \conv(A_i) \neq\emptyset$. Note that if for some coloring $\chi \in [h]^Q$ with color classes $(Q_1,\ldots,Q_h)$, the restriction of $\chi$ to $A$ is $(A_1,\ldots,A_h)$, then $\bigcap_{i=1}^h \conv(Q_i) \supset \bigcap_{i=1}^h \conv(A_i) \neq\emptyset$. Hence, the family $\calB$ of all colorings $(Q_1,\ldots,Q_h)$ such that $\bigcap_{i=1}^h \conv(Q_i)=\emptyset$ does not shatter any set $A$ of size $r_h$, since the set of restrictions $\{\chi|_A:\chi\in\calB\}$ misses $(A_1,\ldots,A_h)$.  
Hence, by Theorem~\ref{thm:qary-sauer} we have 
\[
\frac{|\calB|}{h^m}\leq \frac{1}{h^m} \sum_{i=0}^{r_h-1}\binom{m}{i}(h-1)^{m-i} = \Prob[\operatorname{Bin}(m,1/h)<r_h],
\]
which proves the assertion.
\end{proof}

We shall need a simple Binomial tail estimate.
\begin{lemma}\label{lem:tail}
Let $Z\sim\operatorname{Bin}(n,p)$. If $a\geq 1$, $\eta\in(0,1)$, and $\mathbb{E}[Z]=np\geq 8(a+\log(1/\eta))$, then 
\[
\Prob[Z<a]\leq\eta.
\]
\end{lemma}

\begin{proof}
The assumption gives $a\leq np/8$, so the event $Z<a$ is contained in the event $Z\leq np/2$. As $\mathbb{E}[Z]=np$, Chernoff's inequality gives $\Prob[Z\leq np/2]\leq e^{-np/8}\leq\eta$.
\end{proof}

Now we are ready to prove Proposition~\ref{prop:random-bins}.

\begin{proof}[Proof of Proposition~\ref{prop:random-bins}]
Fix $I\in\binom{[2k]}{h}$. Let $Q$ be the sub-multiset of elements of $P$ colored in one of the colors in $I$. Fix $m \leq |P|$, and condition on $|Q|=m$. For each $m$, under this conditioning the colors of the elements of $Q$ among the $h$ colors in $I$ are independent and uniform in $[h]$. Hence, by Lemma~\ref{lem:coloring-pattern} we have
\[
 \Prob\left[\bigcap_{i\in I}\conv(P_i)=\emptyset\,\middle|\,|Q|=m\right]
 \leq \Prob[\operatorname{Bin}(m,1/h)<r_h].
\]
We clearly have $|Q|\sim\operatorname{Bin}(N,h/2k)$. Therefore, taking expectation over $|Q|$ we obtain
\[
 \Prob \left[\bigcap_{i\in I}\conv(P_i)=\emptyset \right]
 \leq \Prob[\operatorname{Bin}(\operatorname{Bin}(N,h/2k),1/h)<r_h] = \Prob[\operatorname{Bin}(N,1/2k)<r_h],
\]
where the equality is the classical composition identity for Binomial thinning (see~\cite{SteutelVanHarn1979}).  
As the mean of the Binomial random variable in the right hand side is $N/2k$, Lemma~\ref{lem:tail} and the assumption on $N/2k$ imply the assertion~\eqref{Eq:Aux-Prop3.2}.

By linearity of expectation, the expected number of sets of $h$ color classes with empty intersection is at most $\eta\binom{2k}{h}$. Hence, there exists a coloring with at most that many sets of $h$ color classes with empty intersection.
\end{proof}

\paragraph{Step~2: Showing that if only very few sets of $h$ color classes have empty intersection then the convex hulls of some $k$ color classes intersect.}
The following lemma uses the matching argument of
Holmsen~\cite[Lemma 3.1]{Holmsen20}. The idea is to use the colorful Helly theorem to deduce that if no $k$ color classes intersect, then there exist many sets of $h$ color classes whose intersection is empty.
We prove the lemma in the more general setting of families of $2k$ convex sets in a general convexity space; below, we will apply it to the convex hulls of the color classes in the coloring constructed in Step~1. 
\begin{lemma}\label{lem:near-helly}
Let $(X,\calC)$ have Helly number $h\geq 2$ and colorful Helly number $h_c$. Let $\calF=(F_1,\ldots,F_{2k})$ be a finite labeled family of convex sets, with repetitions allowed. Assume $k\geq h$. Call an $h$-subfamily of $\calF$ bad if it has empty intersection, and let $B$ be the number of bad $h$-subfamilies. If
\begin{equation}\label{eq:near-helly-threshold}
 B \leq (4h^2 h_c)^{-h}\binom{2k}{h},
\end{equation}
then some subfamily of $\calF$ of size at least $k$ has nonempty intersection.
\end{lemma}

\begin{proof}
Suppose, for contradiction, that no subfamily of $\calF$ of size at least $k$ has nonempty intersection. Let $\mathcal M$ be a maximal collection of pairwise disjoint bad $h$-subfamilies, and write $s=|\mathcal M|$. 
By maximality, $\calF \setminus \bigcup\mathcal M$ contains no bad $h$-subfamily.
If $|\calF \setminus \bigcup\mathcal M|\ge k$, then since $k \ge h$, Helly's
theorem in $(X,\calC)$ implies that $\calF \setminus \bigcup\mathcal M$ has nonempty intersection,
contrary to our assumption. Therefore $|\calF \setminus \bigcup\mathcal M|<k$,
and consequently, $|\bigcup \mathcal M|= sh>k$.

We now show that this forces more bad $h$-subfamilies than allowed by \eqref{eq:near-helly-threshold}. First, note that we may assume $s \geq h_c$. Indeed, if we had $s<h_c$, then we would trivially have $B\geq |\mathcal{M}|= s>k/h$ and $k<sh<h h_c$, and using $\binom{2k}{h}\leq (2k)^h$ we would get
\[
 \frac{B}{\binom{2k}{h}} > \frac{1}{2h(2k)^{h-1}} > \frac{1}{2h(2hh_c)^{h-1}} \geq (4h^2 h_c)^{-h},
\]
contradicting \eqref{eq:near-helly-threshold}. Thus, we assume $s\geq h_c$. 

Pick any $h_c$ distinct members $M_1,\ldots,M_{h_c}$ of the matching $\mathcal M$. Regard each $M_i$ as a color class consisting of $h$ convex sets. Since each $M_i$ is bad, no color class has nonempty total intersection. By the  colorful Helly theorem in $(X,\calC)$, there is a colorful choice $G_i\in M_i$ such that $\bigcap_{i=1}^{h_c} G_i=\emptyset$. By the Helly theorem in $(X,\calC)$, at most $h$ of these chosen sets already have empty intersection. Adding further sets $G_i$ if necessary, which preserves the empty intersection, we obtain a bad $h$-subfamily whose $h$ sets come from $h$ distinct members of $\mathcal M$. 
(Adding these sets is possible, since $h_c \geq h$, which is seen by taking all color classes equal in the definition of the colorful Helly number).
Thus, every $h_c$-subcollection of $\mathcal M$ produces at least one bad $h$-subfamily whose elements belong to $h$ distinct members of $\mathcal M$.

On the other hand, a fixed bad $h$-subfamily whose elements belong to $h$ distinct members of $\mathcal M$ can be produced from at most $\binom{s-h}{h_c-h}$ choices of an $h_c$-subcollection of $\mathcal M$.
Therefore,
\[
 B \geq \frac{\binom{s}{h_c}}{\binom{s-h}{h_c-h}} = \frac{s!(h_c-h)!}{h_c!(s-h)! } = \frac{\binom{s}{h}}{\binom{h_c}{h}}.
\]
Using the basic inequalities $\binom{h_c}{h}\leq h_c^h$, $\binom{2k}{h}\leq (2k)^h$, and $\binom{s}{h}\geq(s/h)^h$, we get
\[
 \frac{B}{\binom{2k}{h}} \geq \left(\frac{s}{hh_c (2k)}\right)^h > (2h^2 h_c)^{-h},
\]
contradicting \eqref{eq:near-helly-threshold}. 
Therefore, some  subfamily of $\calF$ of size at least $k$ has a nonempty intersection, as asserted.
\end{proof}

Now we are ready to prove Theorem~\ref{thm:main}.

\begin{proof}[Proof of Theorem \ref{thm:main}]
Let $k \geq 2$ and let $N_k=20 k\bigl(r_h+h\log(4h^2 h_c)\bigr)$. We want to show that any set of $N_k$ points in $X$ admits a $k$-Tverberg partition. 

Note that the assertion holds trivially for $k<h$. Indeed, we clearly have $r_k\le r_h$, since one can merge parts of an $h$-Tverberg partition. As $N_k > r_h$, the assertion holds in this case.
Hence, we may assume $k\geq h$.

Let $P$ be a multiset in $X$ with $|P|=N_k$. As $N_k/2k\geq 10 \bigl(r_h+h\log(4h^2 h_c)\bigr)$, the assumption of Proposition \ref{prop:random-bins} holds for $P$, with respect to $\eta=(4h^2 h_c)^{-h}$. Therefore, by the proposition, there exists a $(2k)$-coloring of the  elements of $P$ into color classes $P_1,\ldots,P_{2k}$ such that at most $(4h^2 h_c)^{-h}\binom{2k}{h}$ of the $h$-subfamilies among the convex sets $\conv(P_1),\ldots,\conv(P_{2k})$ have empty intersection.

By Lemma \ref{lem:near-helly}, applied to the labeled family of the convex hulls of the color classes, there is a subfamily of at least $k$ convex hulls with nonempty intersection. Choose $k$ of the corresponding color classes, say $P_{i_1},\ldots,P_{i_k}$.
These $k$ color classes can be completed into a $k$-Tverberg partition of $P$ by distributing all elements lying in the remaining $k$ color classes arbitrarily among the chosen $k$ color classes. This completes the proof.
\end{proof}

\begin{remark}
Our proof shares with the upper bound proof of  P\'alv\"olgyi~\cite{Palvolgyi} the basic idea that a
non-Tverberg partition must avoid suitable Tverberg patterns on smaller
subsets. P\'alv\"olgyi amplifies the small probability of obtaining such
a pattern by using many disjoint blocks independently, and then applies
the fractional Helly theorem. In contrast, we use the variant of the Sauer-Shelah lemma for $h$-partitions to control the restrictions of a random coloring to
all $r_h$-subsets simultaneously, and then we use Lemma~\ref{lem:near-helly}
directly to pass from few bad $h$-subfamilies to $k$ convex hulls with a
common point. 
\end{remark}

\section{Tight Upper Bound on the Tverberg Number in Separable Spaces}\label{sec:upper-separable}

In this section we prove the following upper bound on the Tverberg number in $S_4$ separable convexity spaces in terms of the Radon number. 
\begin{theorem}\label{thm:main-S4}
Let $(X,\calC)$ be an $S_4$ separable convexity space with Radon number $r$ and Helly number $h\geq 2$.
Then for every $k\geq 2$, $r_k \leq 256hrk$.
\end{theorem}
This bound implies Theorem~\ref{thm:intro-radon-S4}, since by Levi's theorem, we have $h<r$. It is tight up to constant factors, by Theorem~\ref{thm:intro-lower}(1).

\paragraph{Overview of the proof of Theorem~\ref{thm:main-S4}.} The general strategy of the proof is similar to the proof strategy of Theorem~\ref{thm:main} (i.e., the upper bound on the Tverberg number in general spaces). Recall that in Theorem~\ref{thm:main}, we considered a set $P \subset X$ of $N_k$ points (for a sufficiently large $N_k$), we colored the points of $P$ in $2k$ colors, independently at random, and showed that there exists a coloring for which the convex hulls of some $k$ color classes have a non-empty intersection. This was done in two steps: First, we used a Sauer-Shelah argument to show that there exists a coloring for which only for a few of the sets of $h$ color classes, the convex hulls have empty intersection. Then, we used a colorful Helly argument to show that for the coloring from the first step, some $k$ of the convex hulls must have a non-empty intersection.

In the proof of Theorem~\ref{thm:main-S4} below, the second step is the same as in the proof of Theorem~\ref{thm:main}. The significant gain comes from a different Sauer-Shelah type argument at the first step, which allows deducing the same conclusion as in Theorem~\ref{thm:main}, for a much smaller value of $N_k$. 

For the sake of simplicity, we assume that $h$ is a power of $2$ (as will be shown at the end of Section~\ref{subsec:counting}, this affects the bound only by a constant factor). We take a set $P$ of $N_k=O(rhk)$ points in $X$, color its points in $2k$ colors, uniformly at random, and consider sets of $h$ color classes. We say that a set $(P_{i_1},\ldots,P_{i_h})$ of color classes is \emph{bad} if $\mathrm{conv}(P_{i_1}) \cap \ldots \cap \mathrm{conv}(P_{i_h})=\emptyset$. We use a \emph{separation tree} method to show that the probability of a randomly chosen set of $h$ color classes to be bad is very small.  

This is done as follows. Restrict ourselves to the set of points $P_{i_1} \cup \ldots \cup P_{i_h}$. First, we show that if $(P_{i_1},\ldots,P_{i_h})$ is bad, then there exists a binary separation tree with the following properties:
\begin{enumerate}
    \item The tree has $\log h$ levels (except the root) and $h$ leaves that correspond to the color classes;
    \item The separation at each node of the tree is performed according to the question, whether $x \in H_i$ for some halfspace $H_i$;
    \item When the separation tree process is applied to any $x \in P_{i_1} \cup \ldots \cup P_{i_h}$, it is routed to a leaf whose color \emph{does not} correspond to the color of $x$. 
\end{enumerate}    
Then, we use a Sauer-Shelah argument to bound from above the number of separation trees that satisfy the properties~(1)--(3). 

The key observation is that when the separation tree process is applied to some point $x$, once at some level of the tree, $x$ was routed to a side that does not contain the leaf corresponding to its color, we don't have to follow it at the subsequent levels. As a result, the number of points that are `relevant' in the $i$'th level of the tree essentially reduces exponentially with $i$. Here we gain the critical $\log h$ factor, which improves the overall bound from $O((rh \log h) k)$ to the sharp $O(rhk)$.  

We note that related separating-tree constructions
appear in the study of combinatorial dimension (e.g.,~\cite[Section~2]{MendelsonVershynin03},~\cite[Section~3]{RudelsonVershynin06}), while recursive
halfspace partitions are classical in convex localization (see~\cite[Section~3]{Klartag17} and the references therein). It seems that such separation trees have not been used so far in Helly-type problems; possibly they can be useful in other problems as well.

\subsection{Constructing a binary separation tree with separation by halfspaces}

The following lemma shows that for every bad set of $h$ color classes we can construct a binary separation tree that satisfies properties~(1)--(3) above. For ease of notation, we denote the color classes by $Q_1,\ldots,Q_h$ and set $Q:=Q_1 \cup \ldots \cup Q_h$.
\begin{lemma}[Separation tree]\label{lem:separation-tree}
Let $(X,\calC)$ be a $S_4$ separable space. 
Let $h=2^\ell$, for some $\ell \in \mathbb{N}$. Let $Q=Q_1\sqcup\cdots\sqcup Q_h$ be a coloring of a finite labelled multiset of elements of $X$ such that
$\bigcap_{i=1}^h\conv(Q_i)=\emptyset$.  Then there is a full binary tree of depth $\ell$, whose leaves are labelled by the $h$ colors, with a
halfspace at every node except for the leaves, such that every point $x \in Q$ is routed by the tree to a leaf whose label is different from its color.
\end{lemma}

\begin{figure}[htbp]
    \centering
    \includegraphics[
    width=0.55\textwidth,
    trim=1cm 7cm 12cm 0cm,
    clip
    ]{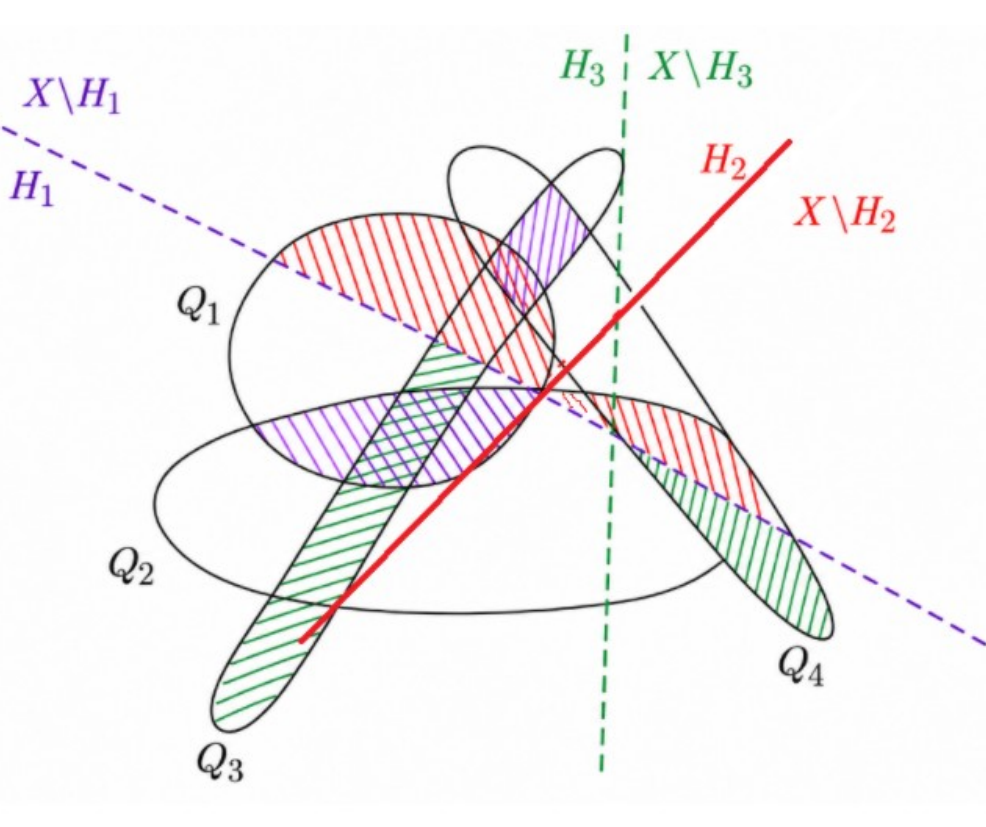}
    \caption{Illustration of the separation tree construction
    in Lemma~\ref{lem:separation-tree} for four color classes $Q_1,\ldots,Q_4$. The separating halfspace at the root is $H_1$; it separates between $Q_1 \cap Q_2$ and $Q_3 \cap Q_4$ (both colored purple). The separating halfspace at the left child is $H_2$; it separates between $(X \setminus H_1)\cap Q_1$ and $(X \setminus H_1) \cap Q_2$ (both colored red). The separating halfspace at the right child is $H_3$; it separates between $H_1\cap Q_3$ and $H_1 \cap Q_4$ (both colored green). A point in $Q_1$ either belongs to $H_1$ and then is routed to the leaves corresponding to colors $\{3,4\}$, or belongs to $(X \setminus H_1) \cap H_2$ and then is routed to the leaf $2$.}
    \label{fig:separation-tree}
\end{figure}

\begin{proof}
At every node, we keep the convex region $D \subset X$ of all points routed to that node and the set $I$ of colors corresponding
to the leaves below that node. We maintain the invariant
\begin{equation}\label{Eq:Disjoint1}
        D\cap\bigcap_{i\in I}\conv(Q_i)=\emptyset.
\end{equation}
It holds at the root, where $D=X$ and $I=[h]$. Suppose a node has children whose color sets are the two halves $I_0$ and $I_1$ of $I$.  By~\eqref{Eq:Disjoint1}, the two convex sets
\[
 D\cap\bigcap_{i\in I_0}\conv(Q_i)
 \qquad\text{and}\qquad
 D\cap\bigcap_{i\in I_1}\conv(Q_i)
\]
are disjoint.  By $S_4$-separability, there is a halfspace $H$ containing
the first and disjoint from the second.  Route the points in
$D\setminus H$ to the child indexed by $I_0$ (i.e., attach the region $D \setminus H$ to that child), and the points in
$D\cap H$ to the child indexed by $I_1$ (i.e., attach the region $D \cap H$ to that child).  Both new regions are convex,
and the invariant remains true for both children. The process is illustrated in Figure~\ref{fig:separation-tree}.

As the process is somewhat complex, we describe in detail the first step. At the root, we have $D=X$ and the sets of children are $I_0=\{1,\ldots,h/2\}$ and $I_1=\{h/2+1,\ldots,h\}$. The convex sets $\bigcap_{i\in I_0}\conv(Q_i)$ and $\bigcap_{i\in I_1}\conv(Q_i)$ are disjoint since $\bigcap_{i\in [h]}\conv(Q_i)=\emptyset$. Using $S_4$ separability, we find a halfspace $H$ such that $\bigcap_{i\in I_0}\conv(Q_i) \subset H$ and $\bigcap_{i\in I_1}\conv(Q_i) \subset X \setminus H$. Then, we put the question ``Is $x \in H$?'' at the root node, and route the points for which the answer is positive to the child that corresponds to $I_1$ -- i.e., attach the region $H$ to that node. At the same time, we route the points for which the answer is negative to the child that corresponds to $I_0$ -- i.e., attach the region $X \setminus H$ to that node. As a result, we have $H \cap \bigcap_{i\in I_1}\conv(Q_i)=\emptyset$ and $(X \setminus H) \cap \bigcap_{i\in I_0}\conv(Q_i)=\emptyset$, and thus, the invariant is satisfied for the leaves corresponding to $I_0$ and $I_1$.

At a leaf labelled $i$, the invariant says that its region is disjoint
from $\conv(Q_i)$. Hence, no point of color $i$ can be routed to that
leaf.
\end{proof}

\subsection{Upper bound on the probability that a set of $h$ colors is bad}\label{subsec:counting}

We now obtain an upper bound on the number of bad sets of $h$ color classes, by bounding from above the number of separating trees that satisfy properties~(1)--(3) and using Lemma~\ref{lem:separation-tree}.

We will need the notion of \emph{VC dimension} and the classical Sauer-Shelah lemma.
A set $W$ of vertices in a hypergraph $\mathcal{H}=(V,\mathcal{E})$ is said to be \emph{shattered} by $\mathcal{E}$ if any subset of $W$ is the trace of some hyperedge -- i.e., if $\{W \cap e:e \in \mathcal{E}\}=2^W$. The \emph{VC dimension} of $\mathcal{H}$ is the maximum size of a vertex set shattered by $\mathcal{E}$. The classical Sauer-Shelah lemma asserts that if the VC-dimension of $\mathcal{H}$ is $d$ and $|V|=n \geq d$, then 
\begin{equation}\label{Sauer-Shelah}
|\mathcal{E}|\leq \sum_{i=0}^d \binom{n}{i} \leq \left(\frac{en}{d} \right)^d.  
\end{equation}
We apply this lemma to the \emph{halfspaces hypergraph} $\mathcal{H}$, whose vertex set is $X$ and whose hyperedges are all halfspaces in $X$, and to its induced sub-hypergraphs. Throughout the proof, the VC dimension of $\mathcal{H}$ is denoted by $d$. It is well known that in any $S_4$ separable space, we have $d=r-1$ (see~\cite[Lemma~3.1]{KellerSmorodinsky}). The VC dimension of any induced sub-hypergraph $\mathcal{H}'=(V',\{e \cap V':e \in \mathcal{E}\})$ is at most $d$.  

\begin{lemma}[Counting bad colorings]\label{lem:count-bad-colorings}
Let $(X,\calC)$ be a $S_4$ separable space. 
Let $h=2^\ell$, for some $\ell \in \mathbb{N}$. Let 
$Q$ be a labeled multiset of $m=16dh^2$ points of $X$. 
Then the number of $h$-colorings
$Q=Q_1\sqcup\cdots\sqcup Q_h$ satisfying
$\bigcap_{i=1}^h\conv(Q_i)=\emptyset$ is at most
$h^m e^{-6.5dh}$. Equivalently, the probability that a random $h$-coloring of $Q$ is bad is at most $e^{-6.5dh}$. 
\end{lemma}

Note that the same bound on the fraction of bad colorings clearly holds for larger multisets, by restricting a bad coloring to any fixed submultiset
of the required size.

\begin{proof}[Proof of the lemma.]
Fix a bad coloring and a separation tree supplied for it by
Lemma~\ref{lem:separation-tree}.  Write the color of a point and the leaf
to which it is routed as binary words of length $\ell$. Note that in these notations, the decision at the $i$'th level of the tree determines the $i$'th bit of the leaf to which $x$ is routed. 

\paragraph{Profiles and route vectors.} We introduce two notions that will play a central role in the proof. For each point $x\in Q$, let
\[
        J(x):=\min\{j:\text{the $j$th bit of the color of $x$
        differs from the $j$th bit of its route}\}.
\]
Thus, $J(x)$ is the first level at which $x$ becomes irrelevant, since it cannot be routed to its color anymore. We call
the vector $\textbf{J}=(J(x))_{x\in Q}$ the \emph{profile} of the pair (coloring, tree).

Given a separation tree $T$ and a profile $\textbf{J}$, for each $x \in Q$ we denote by $R(x)$ the vector consisting of the first $J(x)$ bits of the route of $x$ -- i.e., the route of $x$ at all levels where it was relevant. We call
the vector $\textbf{R}=(R(x))_{x\in Q}$ the \emph{route vector} of the pair (tree, profile). Importantly, the route vector \emph{does not depend on the coloring}.  

For a profile $\textbf{J}=(J(x))_{x\in Q}$, denote by $\textbf{R}(\textbf{J})$ the set of distinct route vectors that correspond to $\textbf{J}$. Furthermore, for a profile $\textbf{J}$ and a route vector $\textbf{R} \in \textbf{R}(\textbf{J})$, denote by  
$C(\textbf{J},\textbf{R})$ the number of colorings that correspond to $(\textbf{J},\textbf{R})$. The number $B$ of bad colorings clearly satisfies
\begin{equation}\label{Eq:B}
B \leq \sum_{\textbf{J}} \sum_{\textbf{R} \in \textbf{R}(\textbf{J})}  C(\textbf{J},\textbf{R}).
\end{equation}
In the rest of the proof, we will bound from above $C(\textbf{J},\textbf{R})$ using a simple counting argument and $|\textbf{R}(\textbf{J})|$ using a Sauer-Shelah argument. The crucial point behind the definition of $\textbf{J}$ and $\textbf{R}$ is that they allow us to discard points at the levels where they become irrelevant, thus leading to a significantly stronger bound in the Sauer-Shelah argument. 

\paragraph{Bounding $C(\textbf{J},\textbf{R})$.}
Fix $(\textbf{J},\textbf{R})$. If $J(x)=j$, then the first $j-1$ bits of the color of $x$ are forced to agree with its route, the $j$'th bit is forced to disagree with its route, and the remaining $\ell-j$ bits are arbitrary.  Hence, for any $\textbf{R}$, the number of distinct colorings corresponding to $(\textbf{J},\textbf{R})$ is at most
\begin{equation}\label{Eq:C(J,R)}
        \prod_{x\in Q}2^{\ell-J(x)}
        =h^m\,2^{-\sum_{x\in Q}J(x)}.        
\end{equation}

\paragraph{Bounding $|\textbf{R}(\textbf{J})|$.}
We now bound the number of distinct \emph{route vectors} corresponding to a fixed profile. Fix $\textbf{J}$.
For $0\leq j< \ell$, let
\[
        S_j:=|\{x\in Q:J(x)>j\}|
\]
be the number of points that are still relevant at level $j$.
There are $2^j$ nodes at that level.  If these nodes contain
$s_1,\ldots,s_{2^j}$ relevant points, with
$\sum_i s_i=S_j$, then by the Sauer--Shelah lemma, applied to the restriction of $\mathcal{H}$ to the relevant points, there are at most 
\[
        \prod_{i=1}^{2^j}
        \left(e\left(1+\frac{s_i}{d}\right)\right)^d
\]
choices of halfspace traces at that level that differ on the set of relevant points. (Specifically, we apply~\eqref{Sauer-Shelah} for $s_i \geq d$ and use the trivial bound $2^{s_i}$ for $s_i<d$). Note that we can treat two halfspaces that differ only on routing points that are already irrelevant at this level as identical, since they lead to the same \emph{route vectors}.  

By the concavity of $\log$, the number of choices is at most
\begin{equation}
\label{Eq:Concavity}
        e^{d2^j}
        \left(1+\frac{S_j}{d2^j}\right)^{d2^j}.
\end{equation}
To simplify the computation, we use the following crude estimate:
for all $x\geq0$ and $a>0$,
\[
        \log(1+x)\leq \log(1+a)+\frac{x}{1+a}.
\]
Applying this to~\eqref{Eq:Concavity} with $a=64\cdot 4^{\ell-j}$, we find that the logarithm of the number of different choices at level $j$ is at most
\[
 d2^j\bigl(1+\log(1+64\cdot 4^{\ell-j})\bigr)
 +\frac{S_j}{1+64\cdot 4^{\ell-j}}.
\]
Summing the first term over the levels gives
\begin{align*}
 \sum_{j=0}^{\ell-1}
 d2^j\bigl(1+\log(1+64\cdot 4^{\ell-j})\bigr)
 \leq
 dh\sum_{s\geq1}2^{-s}
 \bigl(1+\log 128+s\log4\bigr) 
 < 9dh,
\end{align*}
where the first inequality uses the substitution $s:=\ell-j$ and the second inequality holds since $\sum_{s \geq 1}2^{-s}=1$ and $\sum_{s \geq 1}s2^{-s}=2$. Consequently, for each profile $\textbf{J}$, 
we have 
\[
        |\textbf{R}(\textbf{J})| \leq e^{9dh}
        \exp\left(
        \sum_{j=0}^{\ell-1}
        \frac{S_j}{1+64\cdot 4^{\ell-j}}
        \right).
\]
Combining this bound with~\eqref{Eq:B} and~\eqref{Eq:C(J,R)}, we obtain 
\begin{equation}\label{Eq:B2}
\frac{B}{h^m} \leq \frac{1}{h^m}\sum_{\textbf{J}} \sum_{\textbf{R} \in \textbf{R}(\textbf{J})}  C(\textbf{J},\textbf{R})
\leq
e^{9dh}
\sum_{\textbf{J}}
2^{-\sum_xJ(x)}
\exp\left(
\sum_{j=0}^{\ell-1}
\frac{S_j}{1+64\cdot 4^{\ell-j}}
\right).
\end{equation}
\paragraph{Bounding $B$.} The advantage of using the profiles for bounding $B$ is that the right hand side of~\eqref{Eq:B2} factors point by
point. Note that since $S_j=\sum_{x\in Q}{\bf 1}_{\{J(x)>j\}}$, we have $\sum_{j=0}^{\ell-1}
\frac{S_j}{1+64\cdot 4^{\ell-j}}=\sum_{x \in Q} \sum_{j<J(x)} \frac{1}{1+64\cdot 4^{\ell-j}}$. Therefore, the right-hand side of~\eqref{Eq:B2} equals
\begin{equation}\label{Eq:B3}
e^{9dh}
\left[
\sum_{j=1}^\ell
2^{-j}
\exp\left(
\sum_{i=0}^{j-1}
\frac{1}{1+64\cdot 4^{\ell-i}}
\right)
\right]^m.
\end{equation}
It is only left to bound the one-dimensional sum in~\eqref{Eq:B3}. For every $j$, we have
\[
 u_j:=\sum_{i=0}^{j-1}\frac{1}{1+64\cdot 4^{\ell-i}}
 \leq \sum_{s\geq1}\frac1{64\cdot 4^s}
 =\frac1{192}.
\]
Hence, we have $e^{u_j}\leq1+2u_j$, and therefore,
\begin{align*}
\begin{split}
\label{Eq:B4}
\sum_{j=1}^\ell
2^{-j}&
\exp\left(
\sum_{i=0}^{j-1}
\frac{1}{1+64\cdot 4^{\ell-i}}
\right) 
\leq \sum_{j=1}^\ell 2^{-j}
+
2\sum_{j=1}^\ell 2^{-j}
\sum_{i=0}^{j-1}
\frac{1}{1+64\cdot 4^{\ell-i}} \leq
\\& \leq \sum_{j=1}^\ell 2^{-j} + 2\sum_{i=0}^{\ell-1}
\frac{\sum_{j=i+1}^\ell 2^{-j}}
     {1+64\cdot 4^{\ell-i}} 
\leq
1-\frac1h
+
2\sum_{i=0}^{\ell-1}
\frac{2^{-i}}{64\cdot 4^{\ell-i}}
<
1-\frac1h+\frac1{32h}
=
1-\frac{31}{32h},
\end{split}
\end{align*}
where the second inequality follows from changing the order of summation.
Substituting 
this into~\eqref{Eq:B3} and recalling that $m=16dh^2$, we obtain
\begin{equation}\label{Eq:B5}
\frac{B}{h^m}
 \leq
 e^{9dh}\left(1-\frac{31}{32h}\right)^m
 \leq
 \exp\left(9dh-\frac{31m}{32h}\right)
 = e^{-6.5dh}.
\end{equation}
This proves the lemma.
\end{proof}

\paragraph{Values of $h$ that are not powers of $2$.} Lemma~\ref{lem:count-bad-colorings} can be easily modified to handle general values of $h$, at the expense of a factor $4$ increase in the lower bound on $m$, to $64dh^2$.
Indeed,
for arbitrary $h$, let
$q=2^{\lceil\log_2h\rceil}<2h$.  Starting from a uniform $h$-coloring, we can independently retain the original color of each point with probability $h/q$, and otherwise assign it uniformly one of the $q-h$ new colors. The resulting coloring is uniform in $[q]$.  Moreover, each of the
original color classes only shrinks, so a bad $h$-coloring always gives
a bad $q$-coloring. Thus, the fraction of bad $h$-colorings is at most
the fraction of bad $q$-colorings.  Replacing $h$ by $q<2h$ in Lemma~\ref{lem:count-bad-colorings} changes the required number of points (i.e., $16dh^2$) by at most a factor of four.

\subsection{Proof of Theorem~\ref{thm:main-S4}}

Recall that by Lemma~\ref{lem:near-helly}, if in some $2k$-coloring of a set $P \subset X$, the probability of a set of $h$ color classes to be bad is at most $(4h^2 h_c)^{-h}$, then the convex hulls of some $k$ color classes have a non-empty intersection. The following proposition deduces from Lemma~\ref{lem:count-bad-colorings} that this condition holds once $|P| \geq 256kdh$. This will allow us to apply Lemma~\ref{lem:near-helly} to complete the proof of Theorem~\ref{thm:main-S4}. 

\begin{proposition}\label{prop:random-bins-quadratic}
Let $(X,\calC)$ be an $S_4$ separable space, such that the Helly number of $(X,\calC)$ is $h\geq 2$ and the VC dimension of the halfspaces hypergraph of $X$ is $d$. Let $P$ be an $N$-element multiset in $X$, and color each element of $P$ independently and uniformly in one of $2k$ colors. Denote the color classes by $P_1,\ldots,P_{2k}$. If $\frac{N}{2k}\geq 128dh$,
then for every fixed set $I$ of $h$ color classes,
\begin{equation}\label{Eq:Aux-Prop4.4}
    \Prob\left[\bigcap_{i\in I}\conv(P_i)=\emptyset\right]\leq (4h^2h_c)^{-h}.
\end{equation}
Consequently, there exists a coloring for which at most $(4h^2h_c)^{-h}\binom{2k}{h}$ of the sets of $h$ color classes have empty intersection.
\end{proposition}

\begin{proof}
Let $I$ be a fixed set of $h$ colors, and let $Q$ be the set of points whose colors belong to $I$. Then
$\mathbb E|Q|=\frac{hN}{2k} \geq 128dh^2$, 
and thus, by a Chernoff bound we have
\[
        \Prob[|Q|<64dh^2]\leq e^{-16dh^2}.
\]
Once the set $Q$ is fixed, the colors of its elements are
uniform among the $h$ color classes in $I$. Therefore, if $|Q|\geq 64dh^2$, the variant of Lemma~\ref{lem:count-bad-colorings} for a general $h$ gives
\[
        \Prob\left[
        \bigcap_{i\in I}\conv(P_i)=\emptyset
        \,\middle|\,Q
        \right]
        \leq e^{-6.5dh}.
\]
Hence, the probability that $I$ is bad (i.e., the left hand side of~\eqref{Eq:Aux-Prop4.4}) is at most
$e^{-16dh^2}+e^{-6.5dh}$. We claim that this probability is smaller than $(4h^2h_c)^{-h}$. Indeed, in an $S_4$-separable space, $d=r-1$ and $h\leq d$. Theorem~\ref{thm:intro-colorful} gives $h_c\leq h2^{r}=h2^{d+1}$. Thus, we have $h\log(4h^2h_c)\leq4dh$, and therefore, 
\[
        \Prob\left[\bigcap_{i\in I}\conv(P_i)=\emptyset\right]\leq e^{-16dh^2}+e^{-6.5dh} \leq e^{-4dh}
        \leq
        \exp\bigl(-h\log(4h^2h_c)\bigr)
        =(4h^2h_c)^{-h}. \qedhere
\]
\end{proof}

Now we are ready to prove Theorem~\ref{thm:main-S4}. Recall that the theorem asserts that in any $S_4$ separable convexity space with Helly number $h$ and
halfspace VC-dimension $d$, we have $r_k(X,\calC) \leq 256dhk$.

\begin{proof}[Proof of Theorem~\ref{thm:main-S4}]
If $h=1$, the assertion is immediate. Hence, we may assume $h \geq 2$. We consider two cases, according to the size of $k$.

\medskip \noindent \emph{Case~1: $k\geq h$.} Let $P$ be a set of $N_k=256dhk$ points in $X$. Color the points of $P$ in $2k$ colors, independently at random. By Proposition~\ref{prop:random-bins-quadratic} (which can be applied since $N_k/2k = 128dh$) and linearity of expectation, there exists a coloring for which for at most $(4h^2h_c)^{-h}\binom{2k}{h}$ of the sets of $h$ color classes, the convex hulls have empty intersection.  By Lemma~\ref{lem:near-helly}, which can be applied since $k \geq h$, the convex hulls of some $k$ color classes have a non-empty intersection. This yields a Tverberg $k$-partition of $P$, by assigning the remaining points of $P$ to these $k$ classes arbitrarily. 

\medskip \noindent \emph{Case~2: $2 \leq k < h$.}
Let $P$ be a set of $N_k=64dk^2<256dhk$ points in $X$. Color the points of $P$ in $k$ colors, independently at random. We apply the variant of Lemma~\ref{lem:count-bad-colorings} for general values of $h$, but with $k$ in place of $h$ (i.e., considering the set of all $k$ colors). By the lemma, the probability that the convex hulls of the color classes have empty intersection is at most $e^{-6.5dk}<1$. Hence, there exists a coloring for which the convex hulls of the $k$ color classes have a non-empty intersection. The color classes in this coloring form a Tverberg $k$-partition of $P$.

\medskip Finally, in an $S_4$-separable space one has $d=r-1$ and, by Levi's theorem, $h<r$.  Thus, $dh<r^2$, and hence, $r_k(X,\calC) \leq O(r^2k)$, as asserted.
\end{proof}

\section{A Lower Bound on the Tverberg Number}\label{sec:lower}

In this section we prove Theorem \ref{thm:intro-lower} which provides a lower bound on the Tverberg number and the $t$-wise Tverberg number in terms of the Radon number in general convexity spaces. Our construction is a variant of the box convexity on $[q]^n$, considered by P\'alv\"olgyi~\cite{Palvolgyi}. The general idea motivating the construction is to look for a space where in order for a partition of a set of randomly selected points to be Tverberg, the partition will need to satisfy a number of independent conditions, which will be exponentially unlikely.

Fix an integer $n\geq2$, set $q=\lceil\log_2 n\rceil+2$, and let $(X_n,\mathcal{C}_n)$ be the space in which $X_n=[q]^n$ and $\mathcal{C}_n$ consists of all the boxes $A_1\times A_2\times\cdots\times A_n$ with $A_j\subseteq[q]$, see Figure~\ref{fig:box-space-construction}. Boxes are closed under arbitrary intersections and include $\emptyset$ and $X_n$, so this is a finite convexity space. For each $P \subset [q]^n$, we have $\conv(P)=\pi_1(P)\times\cdots\times\pi_n(P)$,
where $\pi_j$ denotes the projection to the $j$th coordinate.

We prove the following, which includes the assertion of Theorem~\ref{thm:intro-lower} as the case $\alpha=1/2$.
\begin{theorem}\label{thm:lower}
Let $0<\alpha<1$ and write $r=r(X_n,\mathcal{C}_n)$. For all sufficiently large $n$ (depending on $\alpha$), the space $(X_n,\mathcal{C}_n)$ is $S_4$ separable, its Radon number satisfies $r=\Theta(\log n)$, and the following holds.
\begin{enumerate}
\item We have
\[
 r_k(X_n,\mathcal{C}_n)\ \geq\ \tfrac{\alpha}{3}\,rk\log k,
 \qquad\text{for every $k$ with } (8\log r)^{1/(1-\alpha)}\leq k\leq n.
\]
Moreover, $r_n(X_n,\mathcal{C}_n)=\Theta(nr^2)$. 

\item For every $2\leq t\leq n$,
\[
 r_{n,t}(X_n,\mathcal{C}_n)=\Theta\bigl(rn\min\{r,t\}\bigr).
\]
In particular, for every constant $t$, we have $r_{n,t}(X_n,\mathcal{C}_n)=\Theta(rn)$.
\end{enumerate}
\end{theorem}
We note that while our argument uses multiset notions, at the end of the proof we will show that the point configurations in the lower bounds may be taken to be sets of distinct points, so the lower bound also holds for the classical (i.e., set-based) Tverberg number.

\begin{figure}[tbp]             
 \centering                           
 \resizebox{0.72\linewidth}{!}{%
\begin{tikzpicture}[x=1cm,y=1cm,font=\small,
  line cap=round,line join=round,
  ambient/.style={draw=black!23,line width=.4pt},
  fiber/.style={draw=black!13,line width=.3pt},
  member/.style={circle,inner sep=0pt,minimum size=3.8pt,
    draw=white,line width=.3pt}
]
\definecolor{boxTeal}{RGB}{37,115,123}
\tikzset{
  boundary/.style={draw=boxTeal,line width=.85pt},
  mesh/.style={draw=boxTeal!55,line width=.4pt}
}

\node[font=\normalsize] at (5.195,5.45)
  {Convex sets in $X_3=[4]^3$};

\foreach \panel/\dx in {0/0,1/6.10}
{
  \begin{scope}[
    shift={(\dx,1.20)},
    x={(.86cm,-.28cm)},
    y={(.57cm,.34cm)},
    z={(0cm,.86cm)}]

    \foreach \a in {0,3}
      \foreach \b in {0,3}
      {
        \draw[ambient] (0,\a,\b) -- (3,\a,\b);
        \draw[ambient] (\a,0,\b) -- (\a,3,\b);
        \draw[ambient] (\a,\b,0) -- (\a,\b,3);
      }

    \ifnum\panel=1
      \foreach \b in {0,1,2}
        \foreach \c in {0,1,2}
          \draw[fiber] (0,\b,\c) -- (3,\b,\c);
    \fi
    \foreach \a in {0,1,2,3}
      \foreach \b in {0,1,2,3}
        \foreach \c in {0,1,2,3}
          \fill[black!25] (\a,\b,\c) circle[radius=.85pt];

    \ifnum\panel=0
      \path[fill=boxTeal,fill opacity=.065]
        (0,0,2)--(2,0,2)--(2,2,2)--(0,2,2)--cycle;
      \path[fill=boxTeal,fill opacity=.065]
        (0,0,0)--(2,0,0)--(2,0,2)--(0,0,2)--cycle;
      \path[fill=boxTeal,fill opacity=.065]
        (2,0,0)--(2,2,0)--(2,2,2)--(2,0,2)--cycle;
      \foreach \a in {0,2}
        \foreach \b in {0,2}
        {
          \draw[boundary] (0,\a,\b) -- (2,\a,\b);
          \draw[boundary] (\a,0,\b) -- (\a,2,\b);
          \draw[boundary] (\a,\b,0) -- (\a,\b,2);
        }
      \draw[mesh]
        (1,0,0)--(1,0,2) (0,0,1)--(2,0,1)
        (2,1,0)--(2,1,2) (2,0,1)--(2,2,1)
        (1,0,2)--(1,2,2) (0,1,2)--(2,1,2);
      \foreach \a in {0,1,2}
        \foreach \b in {0,1,2}
          \foreach \c in {0,1,2}
            \node[member,fill=boxTeal] at (\a,\b,\c) {};
    \else
      \foreach \a in {0,3}
      {
        \path[fill=boxTeal,fill opacity=.085]
          (\a,0,0)--(\a,2,0)--(\a,2,2)--(\a,0,2)--cycle;
        \draw[boundary]
          (\a,0,0)--(\a,2,0)--(\a,2,2)--(\a,0,2)--cycle;
        \draw[mesh]
          (\a,1,0)--(\a,1,2)
          (\a,0,1)--(\a,2,1);
        \foreach \b in {0,1,2}
          \foreach \c in {0,1,2}
            \node[member,fill=boxTeal] at (\a,\b,\c) {};
      }
    \fi
  \end{scope}
}

\node[text=boxTeal] at (2.145,-.05) {$Y_{3,3}$};
\node[text=boxTeal] at (8.245,-.05) {$C$};
\end{tikzpicture}%
 }   
 \caption{Two convex sets in the discrete box space $X_3=[4]^3$. The set on the left figure is $Y_{3,3}=\{1,2,3\}\times\{1,2,3\}\times\{1,2,3\}$ and the set on the right figure is $C=\{1,2,3\}\times\{1,4\}\times\{1,2,3\}$.}                                        
 \label{fig:box-space-construction}
\end{figure} 

\medskip We begin with several lemmas that establish properties of the space $(X_n,\mathcal{C}_n)$ and prove some of the assertions of Theorem~\ref{thm:lower}. 
\begin{lemma}\label{lem:box-separable}
The space $(X_n,\mathcal{C}_n)$ is $S_4$ separable.
\end{lemma}

\begin{proof}
For $j\in[n]$ and $A\subseteq[q]$, the slab $\{x\in X_n:x_j\in A\}$ is a box whose complement $\{x\in X_n:x_j\in[q]\setminus A\}$ is also a box. Hence, slabs are halfspaces. Let $A=\prod_jA_j$ and $B=\prod_jB_j$ be disjoint nonempty boxes. Then $A_j\cap B_j=\emptyset$ for some coordinate $j$, since otherwise choosing $x_j\in A_j\cap B_j$ for every $j$ would produce a common point. The slab $\{x:x_j\in A_j\}$ contains $A$, and its complement contains $B$. If one of the two convex sets is empty, the halfspaces $\emptyset$ and $X_n$ separate them.
\end{proof}

\begin{lemma}\label{lem:box-radon}
The Radon number of $(X_n,\mathcal{C}_n)$ satisfies $q+1\leq r(X_n,\mathcal{C}_n)\leq q+\log_2 n+2$.
\end{lemma}

\begin{proof}
For the lower bound, consider the diagonal points $p_i=(i,i,\ldots,i)$ for $i\in[q]$. For any partition of this $q$-point set into two nonempty parts, the first-coordinate projections of the two parts are disjoint, so the box hulls are disjoint. Hence, the diagonal has no Radon partition and so $r(X_n,\mathcal{C}_n)>q$.

For the upper bound, let $P$ be a multiset with no Radon partition. If some point of $X_n$ appeared twice in $P$, placing the two copies in different parts and distributing the rest arbitrarily would give a Radon partition. Hence, the points of $P$ are distinct. Say that a partition $P=P_1\sqcup P_2$ is separated by coordinate $j$ if $\pi_j(P_1)\cap\pi_j(P_2)=\emptyset$. Since convex hulls are product sets, a partition has disjoint convex hulls if and only if it is separated by some coordinate. As $P$ has no Radon partition, every one of the $2^{|P|-1}-1$ unordered nontrivial partitions of $P$ is separated by some coordinate. For fixed $j \in [n]$ and $v \in [q]$, a partition separated by the $j$'th coordinate has all points of the set $\{x\in P:x_j=v\}$ at the same side of the partition. Since $P=\cup_{v \in [q]} \{x\in P:x_j=v\}$, there are at most $2^{q-1}-1$  partitions separated by the $j$'th coordinate. Therefore, $2^{|P|-1}-1\leq n(2^{q-1}-1)$, and consequently, $|P|\leq q+\log_2 n+1$. Hence, every multiset with at least $q+\log_2 n+2$ points has a Radon partition.
\end{proof}

For $2\leq t\leq q$, let $Y_{n,t}=[t]^n\subseteq X_n$, endowed with the induced box convexity (i.e., the convex sets in $Y_{n,t}$ are all sets of the form $\{Y_{n,t} \cap C:C \in \mathcal{C}_n\}$).

\begin{lemma}\label{lem:box-helly-reay}
The Helly number of $Y_{n,t}$ is exactly $t$.
\end{lemma}

\begin{proof}
To show that the Helly number is at most $t$, let $F$ be a finite family of boxes in $Y_{n,t}$ that has empty intersection. There exists a coordinate in which the corresponding subsets of $[t]$ have empty intersection. For each of the $t$ values, choose one member of the family that does not contain it. This yields $t$ elements of the family with empty intersection. This shows that if all subsets of size $t$ have a non-empty intersection, then so does the entire family.

To show that the Helly number is at least $t$, let $F$ be a family of $t$ sets such that in one fixed coordinate, each box in $F$ misses one value, and in all other coordinates, all boxes in $F$ contain all $t$ values. Then all $t$ elements of $F$ have empty total intersection, while every $t-1$ of them intersect.
\end{proof}

\begin{lemma}\label{lem:box-reay-upper}
For every $2\leq t\leq n$,
\begin{equation}\label{Eq:Reay1}
 r_{n,t}(X_n,\mathcal{C}_n)
 \leq
 \left\lceil n\left(\log n+\log\binom nt+q\log t\right)\right\rceil+1.
\end{equation}
In particular,
\[
 r_{n,t}(X_n,\mathcal{C}_n)=O(n\log n \cdot \min\{t,\log n\}),
 \qquad
 r_n(X_n,\mathcal{C}_n)=O(n\log^2n).
\]
\end{lemma}

\begin{proof}
Let $N \geq \left\lceil n\left(\log n+\log\binom nt+q\log t\right)\right\rceil+1$. Color a labeled multiset $P$ of $N$ points independently and uniformly with colors in $[n]$. We will show that with a positive probability, the convex hulls of any $t$ color classes have a non-empty intersection, which means that the coloring yields a $t$-wise Tverberg partition of $P$.

Fix $I\in\binom{[n]}t$ and a coordinate $j$. For $b\in[q]$, let $F_b$ be the multiset of all points in $P$ whose $j$th coordinate is $b$, and put $s_b=|F_b|$. The probability that some color in $I$ is absent from $F_b$ is at most $t(1-1/n)^{s_b}$. The fibers are disjoint, so the events belonging to distinct $b$'s are independent. Hence, the probability that no coordinate value $b \in [q]$ occurs in every color of $I$ is at most
\[
 \prod_{b=1}^q t(1-1/n)^{s_b}
 =t^q(1-1/n)^N\leq t^qe^{-N/n}.
\]
By a union bound over the $n\binom nt$ pairs $(j,I)$, the probability that there exists a pair $(j,I)$ such that in the $j$th coordinate, no value $b \in [q]$ occurs in all colors of $I$, is at most $n\binom{n}{t} \cdot t^qe^{-N/n} < 1$, where the inequality holds due to the assumption on $N$. We may therefore fix a coloring for which, in every coordinate and for every $I$, some value $b \in [q]$ occurs in all colors of $I$. This means that for each $I$, the convex hulls of the color classes in $I$ intersect. This proves~\eqref{Eq:Reay1}.

Substituting $t=n$ in~\eqref{Eq:Reay1} gives $r_n(X_n,\calC_n)=r_{n,n}(X_n,\calC_n) =O(qn\log n)=O(n\log^2n)$.
Since $\log\binom nt=O(t\log n)$ and $q\log t=O(\log n \log t)$,~\eqref{Eq:Reay1} yields $r_{n,t}(X_n,\calC_n)=O(tn\log n)$. As we clearly have $r_{n,t}(X_n,\calC_n)\leq r_{n,n}(X_n,\calC_n)=O(n \log^2 n)$, it follows that $r_{n,t}(X_n,\mathcal{C}_n)=O(n\log n \cdot \min\{t,\log n\})$, as asserted.
\end{proof}

Now we are ready to prove Theorem~\ref{thm:lower}.

\begin{proof}[Proof of Theorem \ref{thm:lower}]
For any $n \in \mathbb{N}$, the space $(X_n,\mathcal{C}_n)$ is $S_4$ separable by Lemma~\ref{lem:box-separable}. Since $q=\lceil\log_2 n\rceil+2$, Lemma~\ref{lem:box-radon} gives
\begin{equation}\label{eq:box-radon-theta}
 q+1\ \leq\ r(X_n,\mathcal{C}_n)\ \leq\ 2q+1.
\end{equation}
Lemma~\ref{lem:box-reay-upper} proves the asserted upper bounds on $r_n(X_n,\calC_n)$ and $r_{n,t}(X_n,\calC_n)$. We now prove the lower bound assertions.

\medskip \noindent \textbf{Proof of (1).}
Sample a multiset $P$ of $N=\lceil\alpha k(q-1)\log k\rceil$ points $x_1,\ldots,x_N$ independently and uniformly from $X_n$. We show that with positive probability the resulting multiset consists of distinct points and admits no $k$-Tverberg partition.

Fix a partition of the index set $[N]$ into $k$ parts, with sizes $s_1,\ldots,s_k$. Empty parts are allowed here; a partition with an empty part  cannot be a Tverberg partition, and the bound below covers it trivially. Fix also a point $y\in X_n$. By the definition of $(X_n,\calC_n)$, the point $y$ lies in the convex hull of an $s$-point independent uniform sample exactly when every coordinate of $y$ is attained by some sample point. Since the parts use disjoint sets of sample points,
\[
 \Prob\Bigl[y\in\bigcap_{i=1}^k\conv(\text{part }i)\Bigr]
 =\prod_{i=1}^k\Bigl(1-(1-1/q)^{s_i}\Bigr)^{n}.
\]
Write $c=\log\frac{q}{q-1}$, so that $(1-1/q)^s=e^{-cs}$ and $0<c\leq\frac1{q-1}$. The function $s\mapsto\log(1-e^{-cs})$ is concave. By Jensen's inequality, using $\sum_i s_i=N$, we obtain
\begin{align*}
\prod_{i=1}^k\Bigl(1-(1-1/q)^{s_i}\Bigr)^{n} &=\prod_{i=1}^k\bigl(1-e^{-cs_i}\bigr)^{n} 
 \leq\bigl(1-e^{-cN/k}\bigr)^{nk} \\
 &\leq\exp\bigl(-nk\,e^{-cN/k}\bigr)
 \leq\exp\bigl(-nk\,e^{-N/(k(q-1))}\bigr).
\end{align*}
By the choice of $N$, we have
$e^{-N/(k(q-1))}\ \geq\ k^{-\alpha}e^{-1/(k(q-1))}\ \geq\ \tfrac12 k^{-\alpha}$,
and hence, 
\[
\Prob\Bigl[y\in\bigcap_{i=1}^k\conv(\text{part }i)\Bigr]=\prod_{i=1}^k\Bigl(1-(1-1/q)^{s_i}\Bigr)^{n} \leq \exp(-\tfrac{1}{2}nk^{1-\alpha}).
\]
By a union bound, the probability that $P$ admits a Tverberg partition is bounded from above by the sum of these probabilities over all $q^n$ choices of $y$ and all $k^N$ partitions, which is
\[
 \exp\Bigl(n\log q+N\log k-\tfrac12 nk^{1-\alpha}\Bigr)
 \leq\exp\Bigl(n\log q+2k(q-1)\log^2k-\tfrac12 nk^{1-\alpha}\Bigr).
\]
We claim that each of the two positive terms in the exponent is at most $\tfrac18 nk^{1-\alpha}$. For the first term, this follows from $k^{1-\alpha}\geq 8\log r\geq 8\log q$, where the first inequality is the lower end of the assumed range of $k$ and the second inequality follows from~\eqref{eq:box-radon-theta}. For the second term, since $k\leq n$,
\[
 2k(q-1)\log^2k\leq 2n^{\alpha}\cdot k^{1-\alpha}(q-1)\log^2 n
 \leq \tfrac18 nk^{1-\alpha}
\]
once $16\,n^{\alpha}(q-1)\log^2n\leq n$, which holds for all sufficiently large $n$ because $q=O(\log n)$. Hence, the probability that the random multiset $P$ admits a $k$-Tverberg partition is at most 
\[\exp(-\tfrac14nk^{1-\alpha})\leq e^{-2n}<\tfrac14,
\]
where the first inequality holds by the assumption $k^{1-\alpha} \geq 8\log r \geq 8$.
Finally, the probability that the $N$ sampled points are not all distinct is at most $N^2/q^n=o(1)<\tfrac14$, since $N\leq 2nq\log n$ while $q^n\geq3^n$. So, with probability at least $\tfrac12$, the sample is a set of $N$ distinct points with no $k$-Tverberg partition. Consequently, 
\[r_k(X_n,\mathcal{C}_n)>N\geq\alpha k(q-1)\log k\geq\tfrac{\alpha}{3}\,r(X_n,\mathcal{C}_n)k\log k,
\]
where the last inequality uses $q-1\geq\frac{r(X_n,\mathcal{C}_n)-1}2-1\geq\frac {r(X_n,\mathcal{C}_n)}3$ for large $n$, by \eqref{eq:box-radon-theta}.

\medskip Lemma~\ref{lem:box-reay-upper}, with $t=n$, gives $r_n(X_n,\mathcal{C}_n)=O(n\log^2n)$. On the other hand, the lower bound just proved at $k=n$, applied with $\alpha=1/2$, gives $r_n(X_n,\mathcal{C}_n)=\Omega(rn\log n)=\Omega(n\log^2n)$. Together with $r=\Theta(\log n)$, this yields $r_n(X_n,\mathcal{C}_n)=\Theta(n\log^2n)=\Theta(nr^2)$, as asserted.

\medskip \noindent \textbf{Proof of (2).}
Suppose first that $2\leq t\leq q$. Let $Y_{n,t}=[t]^n\subseteq X_n$, and consider the space $(Y_{n,t},\calC_n|_{Y_{n,t}})$, i.e., $Y_{n,t}$ endowed with the induced box convexity. Note that when convex hulls of subsets of $Y_{n,t}$ are computed with respect to $(X_n,\calC_n)$, they remain inside $Y_{n,t}$ and agree with their convex hulls in $(Y_{n,t},\calC_n|_{Y_{n,t}})$. The argument used in the proof of~(1), with $k=n$, $\alpha=1/2$, and $t$ in place of $q$, shows that $Y_{n,t}$ contains $\Omega(tn\log n)$ distinct points with no $n$-Tverberg partition. 
Indeed, one takes a multiset of $N=\lceil\frac12n(t-1)\log n\rceil$ points independently and uniformly from $Y_{n,t}$, and the union-bound estimates above hold for any $2\leq t\leq q=O(\log n)$. If these points had a $t$-wise $n$-Tverberg partition in $X_n$, the convex hulls of the parts would form a family of convex sets in $Y_{n,t}$ in which every $t$ sets have a non-empty intersection. Since $Y_{n,t}$ has Helly number $t$ by Lemma~\ref{lem:box-helly-reay}, all the convex hulls would intersect, a contradiction. Hence, for any $2 \leq t \leq q$, we have $r_{n,t}(X_n,\mathcal{C}_n)=\Omega(tn\log n)$.

If $q\leq t\leq n$, then since the Helly number of $(X_n,\calC_n)$ is $q$ by Lemma~\ref{lem:box-helly-reay}, we have $r_{n,t}(X_n,\mathcal{C}_n)=r_n(X_n,\mathcal{C}_n)=\Theta(n \log^2 n)$, where the second equality holds by part~(1) of the theorem. 

Combining the two range yields
\[
 r_{n,t}(X_n,\mathcal{C}_n)
 =\Omega\bigl(n\log n\min\{t,\log n\}\bigr)
 =\Omega\bigl(r \cdot n\min\{t,r\}\bigr).
\]
Combined with the upper bound proved in Lemma~\ref{lem:box-reay-upper}, this proves the assertion of the theorem.
\end{proof}

\section{Upper Bounds on \texorpdfstring{$t$}{t}-wise Tverberg numbers}\label{sec:reay}

In this section we prove Theorem~\ref{thm:intro-reay-general} which provides upper bounds on the $t$-wise Tverberg number in terms of the Radon number, in general and in separable convexity spaces.

\subsection{An upper bound in general convexity spaces}

In this subsection we prove Theorem~\ref{thm:intro-reay-general}(1). Let us restate it:

\medskip \noindent \textbf{Theorem~\ref{thm:intro-reay-general}(1).} Let $(X,\calC)$ be a convexity space with Radon number $r$. Then for any $2 \leq t \leq k$,
\[
r_{k,t}(X,\calC) = O  \left(\left(r^{\lceil\log_2t\rceil}+\log\binom kt\right) k \right) .
\]    
In particular, $r_{k,2}(X,\calC) = O(k(r+\log k))$.

\medskip 

The proof is closely related to the proof of Theorem~\ref{thm:intro-radon}, but there are two notable differences. On the one hand, due to the weaker assertion compared to Tverberg's theorem (i.e., requiring only that every $t$ convex hulls have a non-empty intersection rather than all the $k$ convex hulls), a much better dependence on $r$ (i.e., the Radon number) can be obtained. Moreover, only the first step of the proof of Theorem~\ref{thm:intro-radon} which applies a Sauer-Shelah type argument is required, while the second step can be replaced by a union bound. On the other hand, the dependence on $k$ (i.e., the number of parts) is worse than in Theorem~\ref{thm:intro-radon}, since where $t$ is smaller than the Helly number $h$, one cannot apply Helly-type arguments (like those used in the second step of the proof of Theorem~\ref{thm:intro-radon}) when $t$-tuples are considered. The short proof below essentially applies the first step of the proof of Theorem~\ref{thm:intro-radon} and then uses a simple union bound.

\begin{proof}[Proof of Theorem~\ref{thm:intro-reay-general}(1)]
Let $N_k$ be a large integer to be chosen below. Color the points of an $N_k$-point labeled multiset independently and uniformly with colors in $[k]$.
Fix $I\in\binom{[k]}{t}$, and let $U_I$ be the set of points receiving a color in $I$.
Conditional on $|U_I|=s$, its induced coloring is uniform with $t$ colors.
Applying Lemma~\ref{lem:coloring-pattern} with $t$ in place of $h$, averaging over $|U_I|$ and using the composition identity for Binomial thinning (i.e., the identity $\mathsf{Bin}(\mathsf{Bin}(N,t/k),1/t)\sim\mathsf{Bin}(N,1/k)$) like in the proof of Proposition~\ref{prop:random-bins}, we obtain
\[
\Pr\left[
\bigcap_{i\in I}\conv(P_i)=\emptyset
\right]
\le
\Pr\left[\mathsf{Bin}(N_k,1/k)<r_t\right].
\]
By a Binomial tail bound (i.e., Lemma~\ref{lem:tail}), for $N_k \geq8k\left(r_t+\log\left(2\binom kt\right)\right)$, the right hand side is at most $(2\binom kt)^{-1}$. Hence, by a union bound over the $\binom kt$ choices of $I$, the probability that there exists $I$ for which  
$\bigcap_{i\in I}\conv(P_i)=\emptyset$ is less than $1$. Thus, there exists a coloring of $P$ such that the convex hulls of every $t$ color classes intersect.  
Therefore, we have $r_{k,t}=O\left(k\left[r_t+\log\binom kt\right]\right)$, which implies the assertion of the theorem for a general $t$ via 
Jamison's bound $r_t\leq r^{\lceil\log_2t\rceil}$. 
In particular, for a 
constant $t$ the result is
$O_t(k[r^{\lceil\log_2t\rceil}+\log k])$, and for $t=2$ it is
$O(k[r+\log k])$.
\end{proof}

\subsection{An upper bound in $S_4$ separable spaces}

In this subsection we prove Theorem~\ref{thm:intro-reay-general}(2). Let us restate it:

\medskip \noindent \textbf{Theorem~\ref{thm:intro-reay-general}(2).} Let $(X,\calC)$ be an $S_4$ separable convexity space with Radon number $r$. Then
\[
r_{k,t}(X,\calC)
 =O\!\left(\min \left\{ r^2,rt\log t, rt+\log\binom{k}{t} \right\} \cdot k \right).
\]
In particular, we have $r_{k,t}(X,\calC)=O(rtk)$ for any $k \leq \exp(O(r))$.

\medskip 


Recall that the examples of Section~\ref{sec:lower} satisfy
$r_{n,t}(X_n,\mathcal C_n)=\Theta(rn\min\{t,r\})$.
This suggests that an `optimal' bound would be  $r_{k,t}=O(rk\min\{t,r\})$.
The assertion does not achieve this bound in general, and instead, it provides three upper bounds.
The first is the
$O(r^2k)$ bound of Theorem~\ref{thm:intro-radon-S4}, which holds since $r_{k,t}(X,\calC) \leq r_{k}(X,\calC)$ for all $t \leq k$. This bound matches the target when $t \geq r$, but misses it by a factor of $O(r/t)$ for $t<r$.
We prove the two other bounds below. 

\subsubsection{Proof of the $O(rtk\log t)$ upper bound}

This bound, which misses the lower bound by a factor of $O(\log t)$, is the strongest one among the three when $t$ is small and $k \gg \exp(O(r))$. In particular, it is sharp up to a constant factor for constant values of $t$.
The proof relies on two components: a variant of the classical $\varepsilon$-net theorem and basic properties of $S_4$ separable spaces. We briefly introduce the required notions.

\paragraph{The $\varepsilon$-net theorem.}
Recall that a set $W$ of vertices in a hypergraph $\mathcal{H}=(V,\mathcal{E})$ is said to be \emph{shattered} by $\mathcal{E}$ if any subset of $W$ is the trace of some hyperedge -- i.e., if $\{W \cap e:e \in \mathcal{E}\}=2^W$. The VC dimension of $\mathcal{H}$ is the maximum size of a vertex set shattered by $\mathcal{E}$. 

An \emph{$\varepsilon$-net} of $\mathcal{H}$ is a set $S$ of vertices that pierces any `heavy' hyperedge -- formally, a set $S$ such that for any $e \in \mathcal{E}$ with $|e| \geq \varepsilon |V|$, we have $e \cap S \neq \emptyset$. The classical $\varepsilon$-net theorem of Haussler and Welzl~\cite{HausslerWelzl} asserts that if the VC dimension of $\mathcal{H}$ is bounded then it admits a small-sized $\varepsilon$-net. In the form we use it below, which is due to Blumer, Ehrenfeucht, Haussler, and Warmuth~\cite{BEHW89} (see also~\cite{KPW92}), the theorem asserts that there exists a constant $c_0$ such that if the VC-dimension of $\mathcal{H}$ is at most $d$, then a uniformly random subset of size
$m_0=\left\lceil c_0\frac{d}{\varepsilon}
        \log\frac{2}{\varepsilon}\right\rceil$
is an $\varepsilon$-net with probability at least $3/4$, provided $|V| \ge m_0$.
We need the following generalization.
\begin{lemma}\label{lem:packed-nets-direct}
If $|V|\geq2m_0k$, then $V$ contains $k$ pairwise disjoint
$\varepsilon$-nets, each of size $m_0$.
\end{lemma}

\begin{proof}
Choose a uniformly random permutation of $V$ and split its first $2m_0k$
elements into $2k$ consecutive blocks of size $m_0$. The marginal distribution of every block is uniform over $m_0$-subsets of $V$, so by the $\varepsilon$-net theorem, each block is an $\varepsilon$-net
with probability at least $3/4$. By linearity of expectation, the expected number of blocks that form an $\varepsilon$-net is at least $3k/2$. Hence, there exists a permutation for which at least $3k/2>k$ blocks form $\varepsilon$-nets. These blocks are disjoint by construction.
\end{proof}
We shall apply the lemma to the \emph{halfspaces hypergraph} $\mathcal{H}=(X,\mathcal{E})$ of $(X,\mathcal{C})$, whose vertex set is $X$ and whose hyperedges are all halfspaces in $\mathcal{C}$. Like in Section~\ref{sec:upper-separable}, we denote the VC dimension of this hypergraph by $d$. 

\paragraph{Properties of $S_4$ separable spaces.} We use two properties of $S_4$ separable spaces.
First, in any $S_4$-separable convexity space, we have $d=r-1$ (see~\cite[Lemma~3.1]{KellerSmorodinsky}). Second, as was observed by Alon and Smorodinsky~\cite[Lemma~4.1]{AlonSmorodinsky}, if convex sets
$C_1,\ldots,C_m$ have empty intersection, then there are halfspaces $H_1,\ldots,H_m$ such that
$C_i\subseteq H_i$ for every $i$ and $\bigcap_iH_i=\emptyset$. 

\medskip Now we are ready to prove the $O(rtk\log t)$ bound.

\begin{proof}
Let $P$ be a labeled multiset of $N$ points, where
$N \ge C d k t\log t$, for a sufficiently large constant $C$.
Consider the hypergraph $\mathcal{H}_P=(P,\mathcal{E}|_{P})$ induced by $\mathcal{H}$ on $P$. 
As observed above, its VC dimension
is at most $d$. By Lemma~\ref{lem:packed-nets-direct}, applied with $\varepsilon=1/t$, the hypergraph $\mathcal{H}_P$ has  $k$ pairwise disjoint $(1/t)$-nets $Y_1,\ldots,Y_k$.
Distribute the remaining points of $P$ arbitrarily to obtain a partition
$P=P_1\sqcup\cdots\sqcup P_k$, with $Y_i\subseteq P_i$ for every $i$. We claim that this forms a $t$-wise $k$-Tverberg partition of $P$.

Suppose, for a contradiction, that  $\bigcap_{i\in I}\operatorname{conv}(P_i)=\emptyset$ for some $I\subseteq[k]$ with $|I|=t$.
By the $S_4$ halfspace separation property recalled above, there are
halfspaces $\{H_i\}_{i\in I}$, such that $\operatorname{conv}(P_i)\subseteq H_i$ for all $i\in I$ and $\bigcap_{i\in I}H_i=\emptyset$. The latter implies that the complementary halfspaces $\{X\setminus H_i\}_{i \in I}$
cover $P$. Hence, for some $j\in I$, we have
\[
  |(X\setminus H_j) \cap P|\ge |P|/t.
\]
Since $Y_j$ is a $(1/t)$-net and $(X \setminus H_j) \cap P$ is a hyperedge of $\mathcal{H}_P$, we have $Y_j \cap (X \setminus H_j) \neq \emptyset$, which is a contradiction since $Y_j \subset P_j \subset \mathrm{conv}(P_j) \subset H_j$. 

Thus, $P$ indeed admits a $t$-wise $k$-Tverberg partition. This proves 
\[
r_{k,t}(X,\mathcal{C})\le O(d k t\log t) \leq O(rkt\log t),
\]
where the second inequality holds since $d=r-1$ in $S_4$ separable spaces.
\end{proof}

\subsubsection{Proof of the $O((rt+\log\binom{k}{t})k)$ upper bound}

This bound, which misses the lower bound by an additive term of $O(kt\log k)$, is the strongest one among the three when $k$ is not extremely large compared to $r$. In particular, it matches the `optimal' bound $O(rtk)$ up to a constant factor for all $k \leq \exp(O(r))$.
The proof uses the binary separation tree method of Section~\ref{sec:upper-separable}.

\medskip We shall use the following variant of Lemma~\ref{lem:count-bad-colorings}.
\begin{lemma}\label{lem:count-bad-colorings-Reay}
Let $(X,\calC)$ be an $S_4$ separable space. Assume that the VC dimension of the halfspaces hypergraph of $(X,\calC)$ is $d$. Let $m \geq 0$ and $t \geq 2$ and let  
$Q$ be a labeled multiset of $m$ points of $X$. 
Then the probability that for a random $t$-coloring $Q=Q_1\sqcup\cdots\sqcup Q_t$  we have $\bigcap_{i=1}^t\conv(Q_i)=\emptyset$ is at most $e^{18dt}\left(1-\tfrac{31}{64t} \right)^m$. 
\end{lemma}

\begin{proof}
Denote the number of $t$-colorings for which the intersection of convex hulls in empty by $B$. First, assume that $t$ is a power of $2$. The proof of Lemma~\ref{lem:count-bad-colorings} can be applied verbatim (with $t$ in place of $h$ and not using any assumption on $m$) to get the following analogue of~\eqref{Eq:B5}: 
\begin{equation*}
\frac{B}{t^m}
 \leq
 e^{9dt}\left(1-\frac{31}{32t}\right)^m.
\end{equation*}
For arbitrary $t\ge2$, let $q=2^{\lceil\log_2t\rceil}$, so that
$t\le q<2t$. As was described in the paragraph after the proof of Lemma~\ref{lem:count-bad-colorings}, a uniform
$t$-coloring can be extended to a uniform $q$-coloring in such a way that every old color class only shrinks. In particular, a
$t$-coloring with an empty intersection gives a $q$-coloring with an empty intersection. Consequently, for every $m \geq 0,t \geq 2$, we have
\[
\Pr[\bigcap_{i=1}^t\conv(Q_i)=\emptyset] \le e^{9dq} \left(1-\frac{31}{32q} \right)^m
 \le e^{18dt}\left(1-\frac{31}{64t}\right)^m. \qedhere
\]
\end{proof}

Now we are ready to prove the upper bound.

\begin{proof}[Proof of the upper bound.]
Let $X$ be an $S_4$ separable space, such that the VC dimension of the halfspace hypergraph of $H$ is $d$. Let $N_k\ge\frac{64k}{31}\left(18dt+\log\left(2\binom kt\right)\right)$. We show that any multiset of $N_k$ points admits a $t$-wise $k$-Tverberg partition. Since in an $S_4$ separable space we have $d=r-1$, this proves the asserted upper bound.

Let $P$ be an $N_k$-element labeled multiset, and color its elements
independently and uniformly with colors in $[k]$, obtaining color
classes $P_1,\ldots,P_k$. Fix $I\in\binom{[k]}t$, and let $U_I$ be the
sub-multiset of points receiving a color in $I$. Then
$|U_I|\sim\operatorname{Bin}(N_k,t/k)$, and once $U_I$ is fixed, its
points are independently and uniformly colored with the $t$ colors in $I$. Thus, applying Lemma~\ref{lem:count-bad-colorings-Reay} for each fixed  $U_I$ and taking expectation over $U_I$, we obtain
\[
 \Pr\left[\bigcap_{i\in I}\conv(P_i)=\emptyset\right]
 \le e^{18dt}\,\mathbb E_{U_I}\left[
          \left(1-\frac{31}{64t}\right)^{|U_I|}\right]
 =e^{18dt}\left(1-\frac{31}{64k}\right)^{N_k}
 \le\exp\left(18dt-\frac{31N_k}{64k}\right).
\]
Here, the equality holds due to the binomial identity
$\mathbb E[a^{\operatorname{Bin}(N_k,p)}]=(1-p+pa)^{N_k}$, which is simply the value of the moment generating function of $\mathrm{Bin}(n,p)$ at the point $t=\ln a$.

By a union bound over all $\binom kt$ choices of $I$, the probability
that some $t$ color-class convex hulls have empty intersection is at most
\[
 \binom kt\exp\left(18dt-\frac{31N_k}{64k}\right) \leq \exp \left(18dt-\frac{31N_k}{64k}+\log \binom{k}{t}\right) \leq 1/2,
\]
where the second inequality holds by the assumption $N_k\ge\frac{64k}{31}\left(18dt+\log\left(2\binom kt\right)\right)$.
Hence, there exists a coloring for which the convex hulls of every $t$ color classes intersect. The color classes of this coloring form a $t$-wise $k$-Tverberg 
partition of $P$.
This completes the proof.
\end{proof}


\section{Applications}\label{sec:applications}

Our improved upper bounds on the colorful Helly and Tverberg numbers in terms of the Radon number in general convexity spaces (i.e., Theorems~\ref{thm:intro-colorful} and~\ref{thm:main}) lead to improved quantitative bounds for several other theorems: the fractional Helly theorem, the selection lemma, the weak $\varepsilon$-net theorem and the $(p,q)$-theorem, in general convexity spaces.
In this section we present the improved results. 
As the proofs (except for the proof of the fractional Helly theorem) consist mostly of substituting our improved bounds into classical leveraging frameworks, the proofs are presented very briefly. 


Throughout this section, we assume that $(X,\mathcal{C})$ is an abstract convexity space with Helly number $h\geq2$ and Radon number $r$. We introduce two notations that will appear in the assertions of the theorems below. The first is the bound of Theorem~\ref{thm:intro-colorful} on $h_c$. Put
\begin{equation}\label{eq:applications-L}
 L:=1+\sum_{t=1}^{h-1}(2^{r-t}-2)(h-t)<2^r h .
\end{equation}
The second notation will appear in the fractional Helly theorem. For $0<\alpha\leq1$, define
\begin{equation}\label{eq:applications-beta}
 \beta_{h,L}(\alpha)
 :=\left(\frac{\alpha}{2}\right)^{h^{L-1}}
   (12hL)^{-\frac{h(h^{L-1}-1)}{h-1}} .
\end{equation}

\paragraph{Fractional Helly theorem.} The fractional Helly number $h_f$ of an abstract convexity space $(X,\mathcal{C})$ is the least integer $q$ with the property that for every $\alpha>0$ there exists $\beta(\alpha)>0$ such that the following holds. If for some finite $\mathcal{F} \subset \mathcal{C}$, an $\alpha$-fraction of the $q$-subfamilies of $\mathcal{F}$ intersect, then some point belongs to $\beta |\mathcal{F}|$ elements of $\mathcal{F}$.
\begin{theorem}[Fractional Helly]\label{thm:applications-fh}
Let $(X,\mathcal{C})$ be an abstract convexity space with Radon number $r$ and Helly number $h$, and let $L,\beta_{h,L}(\alpha)$ be as defined in~\eqref{eq:applications-L},~\eqref{eq:applications-beta}, respectively.
Let $\calF$ be a finite labeled family of at least $L$ convex sets. If at
least an $\alpha$-fraction of its $L$-subfamilies intersect, then some
point belongs to at least a $\beta_{h,L}(\alpha)$-fraction of the sets in $\calF$.
In particular, $h_f(X,\mathcal{C}) \leq L$.
\end{theorem}

\begin{proof}
For $\alpha=1$ the assertion follows directly from Helly's theorem, so
assume $0<\alpha<1$. Apply the quantitative form of Holmsen's
forbidden-substructure theorem~\cite{Holmsen20}, through the reduction of
Holmsen and Lee~\cite{HolmsenLee}, to the $h$-uniform intersection
hypergraph of $\calF$. Theorem~\ref{thm:intro-colorful} excludes the
required colorful obstruction. With $f(x)=(x/(12hL))^h$ and $g_{\ell}(x)$ defined as the $\ell$-fold iteration of $f$ (i.e., 
$g_{\ell}(x)=f \circ\ldots \circ f(x)$), their argument
produces, for all sufficiently large families, a clique on a
$g_{L-1}(\alpha)$-fraction of the vertices. Helly's theorem turns this clique
into a subfamily with common intersection.

It remains to remove the ``sufficiently large'' assumption on $|\mathcal{F}|$. To this end, 
replace every member of $\mathcal{F}$ by 
$u$ labeled copies. If $n=|\calF|\geq L(L-1)$, the proportion of
intersecting $L$-subfamilies in the blow-up tends to at least
$\alpha n(n-1)\cdots(n-L+1)/n^L\geq\alpha/2$. Hence, the resulting common-intersection
subfamily has relative size at least
\[
 g_{L-1}(\alpha/2)
 =
 \left(\frac{\alpha}{2}\right)^{h^{L-1}}
 (12hL)^{-\frac{h(h^{L-1}-1)}{h-1}}
 =\beta_{h,L}(\alpha).
\]
If $L\leq n<L(L-1)$, one intersecting $L$-subfamily already gives a
larger relative bound.
\end{proof}

\paragraph{Selection lemma.} A selection lemma asks for a point contained in the convex hulls of a
positive proportion of the subsets of a prescribed size.  Set
\begin{align}
A&:=\left\lceil20L\left(
r^{\lceil\log_2h\rceil}+h\log(4h^2L)\right)\right\rceil,
\label{eq:applications-A}\\
\delta_0&:=(eL)^{-A \cdot L},\qquad
\lambda_0:=\beta_{h,L}(\delta_0),\qquad
C_{r,h}:=\max\{A \cdot L,e^A/\lambda_0\}.
\label{eq:applications-constants}
\end{align}
By Theorem~\ref{thm:main} and Jamison's bound $r_k \leq r^{\lceil \log_2 k \rceil}$, we have $r_L\leq A$. Note that since $L \leq h2^r$, we have
\[A=O\left(2^r h\left[
r^{\lceil\log_2h\rceil}+h(r+\log h)\right]\right)
\leq2^{r+O((\log r)^2)}.\]

\begin{lemma}[Selection lemma]\label{lem:applications-selection}
Let $(X,\mathcal{C})$ be an abstract convexity space with Radon number $r$ and Helly number $h$, and let $L, A, \lambda_0$ be as defined in~\eqref{eq:applications-L},~\eqref{eq:applications-A},~\eqref{eq:applications-constants}, respectively. If $P$ is a labeled set of $n\geq A \cdot L$ points, then some point belongs to the convex
hulls of at least $\lambda_0\binom nA$ of its $A$-subsets.
\end{lemma}

\begin{proof}
Every $(A \cdot L)$-subset contains an $L$-Tverberg partition whose parts
can be enlarged to size $A$, and therefore yields an intersecting
$L$-tuple of convex hulls of $A$-subsets. Thus, at least
\[
 \frac{\binom n{A \cdot L}}{\binom{\binom nA}L}
 \geq(eL)^{-A \cdot L}=\delta_0
\]
of these $L$-tuples intersect.  Theorem~\ref{thm:applications-fh}
gives the assertion.
\end{proof}

\paragraph{Weak $\varepsilon$-net theorem.} For a finite labeled point set $P$ and $0<\varepsilon\leq1$, a weak
$\varepsilon$-net is a set of points in $X$ meeting every convex set
that contains at least $\varepsilon|P|$ labeled members of $P$.

\begin{theorem}[Weak $\varepsilon$-nets]\label{thm:applications-nets}
Let $(X,\mathcal{C})$ be an abstract convexity space with Radon number $r$ and Helly number $h$, and let $A, C_{r,h}$ be as defined in~\eqref{eq:applications-A},~\eqref{eq:applications-constants}, respectively. Every finite labeled point set has a weak $\varepsilon$-net of size at
most
\[
 C_{r,h}\varepsilon^{-A} \leq C_{r,h}\varepsilon^{-2^{r+O((\log r)^2)}}.
\]
\end{theorem}

\begin{proof}
Apply the standard selection-and-greedy argument of
Alon--B\'ar\'any--F\"uredi--Kleitman~\cite{ABFK} to the distinct convex hulls
of traces $P\cap C$ with $C\in\calC$ and $|P\cap C|\geq\varepsilon|P|$.
If $\varepsilon|P|\geq A \cdot L$, Lemma~\ref{lem:applications-selection}
allows every selected point to be charged to at least
$\lambda_0\binom{\lceil\varepsilon|P|\rceil}{A}$ previously uncharged
$A$-subsets, so at most
\[
 \frac{\binom{|P|}{A}}
 {\lambda_0\binom{\lceil\varepsilon|P|\rceil}{A}}
 \leq \frac{e^A}{\lambda_0} \cdot \varepsilon^{-A}
\]
points are selected.  If $\varepsilon|P|<A \cdot L$, the support of $P$ itself
has size at most $A\cdot L \cdot \varepsilon^{-A}$.  The estimate for $A$ follows
from \eqref{eq:applications-L}, $h<r$, and
$\log(hL)=O(r+\log h)$.
\end{proof}

\paragraph{$(p,q)$ theorem.}
For a finite family $\calF$ of nonempty convex sets, its transversal number $\tau(\calF)$ is the minimum number of points meeting every member
of $\calF$; the family has the $(p,q)$ property if among every $p$
distinct members some $q$ have nonempty intersection.

\begin{theorem}[The $(p,q)$ theorem]\label{thm:applications-pq}
Let $(X,\mathcal{C})$ be an abstract convexity space with Radon number $r$ and Helly number $h$, and let $L, \beta_{h,L}, A, C_{r,h}$ be as defined in~\eqref{eq:applications-L},~\eqref{eq:applications-beta},~\eqref{eq:applications-A},~\eqref{eq:applications-constants}, respectively. Suppose $p\geq q\geq L$ and $\calF$ has the $(p,q)$ property. Set
$p'=(p-1)(q-1)+1$ and $\alpha_{p,q}:=\frac{\binom qL}{\binom{p'}L}$.
Then
\[
 \tau(\calF)\leq
 C_{r,h}\beta_{h,L}(\alpha_{p,q})^{-A}
 \leq
 C'_{r,h}
 \left(\frac{\binom{p'}L}{\binom qL}\right)^{Ah^{L-1}}.
\]
\end{theorem}

\begin{proof}
This is the standard fractional-matching argument of
Alon--Kleitman~\cite{AlonKleitman}; see also~\cite{AKMM}.
Note that after duplicating members, among every $p'$ labeled copies some $q$ intersect. Indeed,  either they represent at least $p$ distinct sets, or one
original member occurs at least $q$ times. Double counting therefore
shows that at least an $\alpha_{p,q}$-fraction of all $L$-tuples
intersect. Theorem~\ref{thm:applications-fh}, applied to rational
duplications of a fractional matching, gives an upper bound of
$\beta_{h,L}(\alpha_{p,q})^{-1}$ on the fractional matching number $\nu^*(\calF)$; linear-programming
duality gives the same bound for the fractional transversal number.
Applying Theorem~\ref{thm:applications-nets} to an optimal fractional
transversal then gives the first inequality, and the second follows
from \eqref{eq:applications-beta}.
\end{proof}

For comparison, the Holmsen--Lee argument used a colorful Helly upper bound of
$r^{r^{O(\log r)}}$~\cite{HolmsenLee}.  Theorem~
\ref{thm:applications-fh} replaces it by $L<2^r h$, and the combination
with Theorem~\ref{thm:main} yields the exponent $A$ above, which is much
smaller than that obtained by applying the general strategy of Alon et al.~\cite{AKMM} to the bound of Holmsen and Lee~\cite{HolmsenLee}.  In the $(p,q)$ theorem this also reduces the admissible
threshold to $q\geq L=O(2^r r)$.  We note that Moran and Yehudayoff give a different
weak-net estimate for finite $S_3$-separable spaces
\cite{MoranYehudayoff}; the assumptions and parameter tradeoffs are
different so their result is incomparable with ours.

\section{Concluding remarks}\label{sec:conclusion}

By Theorems \ref{thm:intro-radon} and \ref{thm:intro-lower}, the optimal coefficient $c(r)$ in the linear bound $r_k\leq c(r)k$ satisfies
\[
 \Omega(r^2)\ \leq\ c(r)\ \leq\ O(r^{\lceil \log_2 r \rceil}).
\]
Determining the exact growth seems to be a difficult problem. Here, we suspect that our upper bound is closer to the truth, and propose that the the following 
conjecture is a reasonable next target.

\begin{conjecture}\label{conj:superpoly}
If $r_k\leq c(r)\,k$ holds for every convexity space with Radon number $r$ and every $k\geq2$, then $c(r)$ grows superpolynomially.
\end{conjecture}

For separable spaces the problem is nearly resolved. Under the $S_4$ separation assumption, Theorem~\ref{thm:intro-radon-S4} yields the upper bound $O(r^2k)$ and Theorem~\ref{thm:intro-lower} proves its optimality. Under the essentially weaker $S_3$ assumption, which only asks that every point outside a convex set be separated from it by a halfspace, Cho, Holmsen, Jung, and Liu \cite{CHJL} obtained the upper bound $O(r^2\log r)k$, which is $O(\log r)$ far from the lower bound of Theorem~\ref{thm:intro-lower}. 
We conjecture that the lower bound construction is optimal.
\begin{conjecture}\label{conj:s3}
There is an absolute constant $C$ such that every $S_3$-separable convexity space satisfies $r_k\leq Cr^2k$ for every $k\geq2$.
\end{conjecture}

Theorem~\ref{thm:intro-colorful} shows that every convexity
space with Radon number $r\geq3$ satisfies $h_c<2^r(r-1)$. On the other hand,
P\'alv\"olgyi's box-convexity examples give spaces with
$h_c\geq2^{r-1}-1$ \cite{Palvolgyi}. Thus, the colorful Helly number is
determined up to a factor $O(r)$. We suspect that this factor is not necessary.
\begin{conjecture}\label{conj:strong-colorful}
    There is an absolute constant $C$ such that every convexity space with Radon number $r$ has colorful Helly number at most $C2^r$.
\end{conjecture}

An interesting potentially easier question in this vein is to bound the {\it weak} colorful Helly number. In this setting, we require that the $k$ color classes be sufficiently large with respect to $r$, and only ask for a $(k+1)$-set of them to intersect. This has already been studied in several applications (e.g.,~\cite{BaranyMatousek,FranklJung,FranklJungTomon,GoaocHolmsenPatakova,HolmsenPatakova}), and it would also be interesting to decide whether one can prove an $O(2^r)$ bound in this case. More broadly, it is sometimes the case that the weak colorful Helly number is significantly smaller than the strong one; an improved understanding when and why this happens also seems interesting.

Finally, it may also be interesting to ask whether our arguments for bounding from above the Tverberg number can be used to construct efficient algorithms for constructing Tverberg partitions, a problem that has been studied extensively in the Euclidean setting~\cite{Har-PeledZ21,MillerS10,MulzerW13,RolnickS16}. In the abstract setting, there could be several reasonable computational models in which algorithms can be studied (see~\cite{RolnickS16}). One model, presented by Avis and Houle~\cite{AvisH95}, assumes access to an oracle, which, given $h$ convex sets, decides whether their intersection is nonempty, and potentially returns a point in the intersection if so. In this framework, one can use the method of conditional expectations to derive from our Theorems~\ref{thm:intro-radon} and~\ref{thm:intro-radon-S4} a deterministic algorithm of runtime $n^{O(r^3)}$ in $S_4$ spaces and $n^{r^{O(\log r)}}$ in general spaces. While the runtimes of these algorithms are very high, their interesting feature is that they use only combinatorial properties of the underlying spaces and not geometric properties. It would be interesting to decide if any of these can be improved.


\section*{Acknowledgements}
We are deeply grateful to Attila Jung for numerous helpful suggestions and feedback.

\section*{AI Disclosure}
Claude Opus 5 was used by the second author to check and generate drafts of early versions of parts of the argument in Section~\ref{sec:upper}. In an early draft, Opus 5 pointed the second author to the Karpovsky--Milman lemma~\cite{KarpovskyMilman}, replacing an incorrect use of a different variant of the Sauer-Shelah lemma. ChatGPT 5.6 and ChatGPT 6 were used by the first author for proving Theorem~\ref{thm:intro-radon-S4} and some lemmas and in editing earlier drafts of the manuscript. ChatGPT 6 was used by the second author to study algorithmic applications of our results and by the first two authors to generate the figures in the manuscript.

\end{document}